\documentclass[11pt,a4paper]{article}

\usepackage[margin=1in]{geometry}
\usepackage{amsmath, amssymb, amsthm, mathtools, empheq}
\usepackage{aliascnt, appendix}
\usepackage{enumitem}
\usepackage{xcolor}
\usepackage{booktabs}
\usepackage{multirow}
\usepackage{caption}
\usepackage[colorlinks=true,linkcolor=black,urlcolor=blue,citecolor=black]{hyperref}
\usepackage{cleveref}

\usepackage[
  backend=biber,
  style=numeric-comp,
  maxnames=99,
  minnames=99,
  sorting=nyt,
  url=false,
  isbn=false
]{biblatex}
\usepackage[bblfile=accelerated]{biblatex-readbbl}

\DeclareBibliographyDriver{book}{
  \printnames{author}\addcolon\space
  \printfield{title}
  \iffieldundef{edition}{}{\addcomma\space\printfield{edition}}
  \iffieldundef{series}{}{\newunit}
  \printfield{series}
  \newunit
  \printlist{publisher}\space
  \printfield{year}
  \iffieldundef{note}{}{\newunit}
  \printfield{note}
  \iffieldundef{doi}{}{\newunit}
  \printfield{doi}
  \finentry
}
\DeclareBibliographyDriver{article}{
  \printnames{author}\addcolon\space
  \printfield{title}
  \newunit
  \printfield{journaltitle}\addperiod\space
  \printfield{volume}\printfield{number}\addcomma\addspace
  \printfield{pages}\space
  \printfield{year}
  \iffieldundef{note}{}{\newunit}
  \printfield{note}
  \iffieldundef{doi}{}{\newunit}
  \printfield{doi}
  \iffieldundef{url}{}{\newunit}
  \printfield{url}
  \finentry
}
\DeclareBibliographyDriver{inproceedings}{
  \printnames{author}\addcolon\space
  \printfield{title}\addperiod\space
  In\addspace\printfield{booktitle}\addperiod\space
  \printfield{pages}\space
  \printfield{year}
  \iffieldundef{doi}{}{\newunit}
  \printfield{doi}
  \finentry
}
\DeclareBibliographyDriver{misc}{
  \printnames{author}\addcolon\space
  \printfield{title}
  \newunit
  \printfield{howpublished}\space
  \printfield{year}
  \iffieldundef{doi}{}{\newunit}
  \printfield{doi}
  \finentry
}
\DeclareBibliographyDriver{thesis}{
  \printnames{author}\addcolon\space
  \printfield{title}
  \newunit
  \printfield{type}\addcomma\space
  \printlist{institution}\space
  \printfield{year}
  \iffieldundef{url}{}{\addperiod\space Available at \printfield{url}}
  \finentry
}
\DeclareNameAlias{default}{family-given}
\DeclareFieldFormat{year}{(#1)}
\DeclareFieldFormat[article]{title}{#1}
\DeclareFieldFormat[article]{pages}{#1}
\DeclareFieldFormat[article]{number}{(#1)}
\DeclareFieldFormat[inproceedings]{title}{#1}
\DeclareFieldFormat[inproceedings]{pages}{pp.~#1}
\DeclareFieldFormat[thesis]{title}{\mkbibitalic{#1}}
\DeclareFieldFormat{doi}{\href{https://doi.org/#1}{\nolinkurl{https://doi.org/#1}}}
\DeclareFieldFormat{url}{\href{#1}{\nolinkurl{#1}}}
\makeatletter
\AddToHook{cmd/appendix/before}{\def\cref@section@alias{appendix}\def\cref@subsection@alias{appendix}}
\makeatother

\theoremstyle{definition}
\newtheorem{dfn}{Definition}[section]
\newtheorem{rem}{Remark}[section]
\newtheorem{eg}{Example}[section]
\theoremstyle{plain}
\newtheorem{thm}{Theorem}[section]
\newaliascnt{prop}{thm}
\newtheorem{prop}[prop]{Proposition}
\aliascntresetthe{prop}
\newaliascnt{lem}{thm}
\newtheorem{lem}[lem]{Lemma}
\aliascntresetthe{lem}
\newaliascnt{cor}{thm}
\newtheorem{cor}[cor]{Corollary}
\aliascntresetthe{cor}

\crefname{thm}{Theorem}{Theorems}
\crefname{claim}{Claim}{Claims}
\crefname{lem}{Lemma}{Lemmata}
\crefname{prop}{Proposition}{Propositions}
\crefname{proposition}{Proposition}{Propositions}
\crefname{cor}{Corollary}{Corollaries}
\crefname{dfn}{Definition}{Definitions}
\crefname{rem}{Remark}{Remarks}
\crefname{table}{Table}{Tables}
\crefname{section}{Section}{Sections}
\crefname{chapter}{Chapter}{Chapters}
\crefname{appendix}{Appendix}{Appendices}
\crefname{figure}{Fig.}{Figs.}
\crefname{eg}{Example}{Examples}
\crefname{equation}{Eq.}{Eqs.}

\newcommand{\RR}{{\mathbb{R}}}

\newcommand{\NN}{{\mathbb{N}}}
\newcommand{\rmd}{{\mathrm{d}}}
\newcommand{\rme}{{\mathrm{e}}}

\newcommand{\argmin}{\mathop{\rm arg~min}\limits}
\newcommand{\dom}{\mathop{\rm dom}}
\newcommand{\conv}{\mathop{\rm Conv}}

\newcommand{\diag}{\mathop{\rm diag}}

\title{{\Large Nesterov's Accelerated Gradient for Unbounded Convex Functions\\Finds the Minimum-Norm Point in the Dual Space}}
\author{{\Large Keiya Sakabe}\smallskip\\Faculty of Computer Science,\\Ruhr University Bochum, Bochum, Germany\\\texttt{keiya.sakabe@rub.de}}
\date{}

\begin{document}
\maketitle

\begin{abstract}
    We study the behavior of first-order methods applied to a lower-unbounded convex function $f$, i.e., $\inf f = -\infty$.
    Such a setting has received little attention since the trajectories of gradient descent and Nesterov's accelerated gradient method diverge.
    In this paper, we establish quantitative convergence results describing their speeds and directions of divergence, with implications for unboundedness judgment.
    A key idea is a relation to a norm-minimization problem in the dual space: minimize $\|p\|^2/2$ over $p \in \dom f^\ast$, which can be naturally solved via mirror descent by taking the Legendre--Fenchel conjugate $f^\ast$ as the distance-generating function.
    It then turns out that gradient descent for $f$ coincides with mirror descent for this norm-minimization problem, and thus it simultaneously solves both problems at $\mathcal{O}(k^{-1})$.
    This result admits acceleration; Nesterov's accelerated gradient method, without any modifications, simultaneously solves the original minimization and the dual norm-minimization problems at $\mathcal{O}(k^{-2})$, providing a quantitative characterization of divergence in unbounded convex optimization.
\end{abstract}


\paragraph{MSC-class} 90C25, 90C30, 90C46

\section{Introduction}

Smooth convex minimization is one of the most fundamental problems in optimization:
\begin{equation}\label{intro:primal}
    \text{minimize }\ f(x) \qquad \text{s.t. }\ x \in \RR^n,
\end{equation}
where $f: \RR^n \to \RR$ is convex and $L$-smooth, i.e., its gradient $\nabla f$ is $L$-Lipschitz continuous.
In most studies, $f$ is assumed to be bounded from below, i.e., $\inf f > -\infty$.
In practice, however, $f$ can be \textit{lower-unbounded}, i.e., $\inf f = -\infty$; see \cref{subsec:related} for motivating examples.
When tackling such \textit{possibly unbounded} problems, one either needs to verify boundedness before attempting optimization, or run an optimization algorithm and observe whether it appears to converge.
However, despite its occurrence in practice, unbounded convex objectives have received little attention; see, e.g., \cite{Auslender1997,HS2024,Obuchowska2004}.

In this work, we study the behavior of minimization algorithms, particularly gradient descent and Nesterov's accelerated gradient method~\cite{Nesterov1983}, when applied to \textit{lower-unbounded} convex functions.
Such a setting might seem uninteresting, since the algorithms do not converge, producing diverging sequences.
However, we establish quantitative \textit{convergence} results on such diverging behavior.
In particular, this paper answers the following question:
\begin{center}
    ($\star$) \quad In which direction and at what speed do these algorithms diverge?
\end{center}
Our answer to this question leads to iteration complexity bounds for detecting the unboundedness of $f$.
We further show that the accelerated gradient method can detect it faster than gradient descent.
Moreover, our results are not limited to unboundedness detection; these algorithms find a \textit{certificate of unboundedness}, and the accelerated gradient method again finds it faster.

The answer to question ($\star$) is characterized by the minimum-norm point $p^\star$ of the domain of the Legendre--Fenchel conjugate $\dom f^\ast$ (or of its closure $\overline{\dom f^\ast}$).
A key observation is a relation between problem~\eqref{intro:primal} and another fundamental problem, \textit{norm minimization in a convex set}, explained as follows.
Let us first consider the case where the objective function $f$ is bounded from below, i.e., $\inf f > -\infty$.
Then, under some reasonable assumptions such as the {\L}ojasiewicz inequality, it holds that $f(x) \approx \inf f \iff \|\nabla f(x)\|^2 \approx 0$.
From this point of view, the minimization problem~\eqref{intro:primal} can also be regarded as a norm-minimization problem in the gradient set $\nabla f(\RR^n)$.
In the unbounded case, this relation can be seen in a more significant way; if the set $\nabla f(\RR^n)$ is distant from the origin, then $f$ must be lower-unbounded.
Based on these observations, using the fact that $\overline{\nabla f(\RR^n)} = \overline{\dom f^\ast}$, we consider the following norm-minimization problem:
\begin{equation}\label{intro:dual}
    \text{minimize }\ \|p\|^2/2 \qquad \text{s.t. }\ p \in \dom f^\ast.
\end{equation}
The relation to problem~\eqref{intro:primal} is then summarized as follows: if $\inf f > -\infty$, then the optimal solution of \eqref{intro:dual} is $p^\star = 0$.
Consequently, if the minimum of \eqref{intro:dual} is attained at $p^\star \neq 0$, then such $p^\star$ works as a certificate of unboundedness of $f$.
Particularly if $\dom f^\ast$ is closed, then \eqref{intro:dual} always has an optimal solution $p^\star$ and the following equivalence holds: $\inf f = -\infty \iff p^\star \neq 0$.

While problem~\eqref{intro:dual} is usually considered a constrained problem, it can be solved as an \textit{unconstrained} problem via mirror descent~\cite{NY1983} by taking $f^\ast$ as the distance-generating function.
Moreover, Hirai and Sakabe~\cite{HS2024} recently pointed out that, roughly speaking, ``gradient descent for problem~\eqref{intro:primal} = mirror descent for problem~\eqref{intro:dual}'':
\begin{prop}[\cite{HS2024}, informal]
    Let $(x_k)$ be the trajectory of gradient descent for \eqref{intro:primal} and $(X_k)$ be the trajectory of mirror descent for \eqref{intro:dual}.
    Under appropriate correspondence in initial points and step sizes, it holds that $X_k = \nabla f(x_k)$.
\end{prop}
Although it was not explicitly written in \cite{HS2024}, from this proposition, a convergence rate of $\nabla f(x_k)$ is obtained;
if problem~\eqref{intro:dual} has an optimal solution $p^\star$, then from the $\mathcal{O}(k^{-1})$ convergence of mirror descent~\cite{BBT2017, LFN2018}, $\nabla f(x_k)$ converges to $p^\star$ at $\mathcal{O}(k^{-1})$ in the sense that $\|\nabla f(x_k) - p^\star\|^2 = \mathcal{O}(k^{-1})$.
This generalizes the well-known fact that $\|\nabla f(x_k)\|^2 = \mathcal{O}(k^{-1})$ when $f$ is bounded from below (e.g., \cite[(1.2.22)]{Nesterov2018}).
Furthermore, summing up $-(x_{i + 1} - x_i)/\eta = \nabla f(x_i) \to p^\star$ ($\eta$: step size) from $i = 0$ to $k - 1$ yields convergence of the total displacement: $-(x_k - x_0)/(k\eta) \to p^\star$ at $\mathcal{O}(k^{-1})$.
This answers question ($\star$) in the gradient descent case; $(x_k)$ diverges in the direction of $-p^\star/\|p^\star\|$ at the speed of $\eta\|p^\star\|$.

We next discuss acceleration of the story above.
Nesterov~\cite{Nesterov1983} proposed the following minimization algorithm for problem~\eqref{intro:primal} with bounded objectives, known as Nesterov's accelerated gradient (NAG) method:
\begin{equation}\label{intro:NAG}
    \begin{gathered}
            x^{(k + 1)} \coloneqq y^{(k)} - \frac1L\nabla f\left(y^{(k)}\right), \qquad y^{(k + 1)} \coloneqq x^{(k + 1)} + \frac{\alpha_k - 1}{\alpha_{k + 1}}\left(x^{(k + 1)} - x^{(k)}\right).\\
        (\alpha_0 \coloneqq 1,\quad \alpha_{k + 1} \coloneqq \frac12\left(1 +\sqrt{1 + 4\alpha_k^2}\right))
    \end{gathered}
\end{equation}
A significant advantage of this method is that it achieves a convergence rate of $f(x^{(k)}) - \min f = \mathcal{O}(k^{-2})$, which is known to be optimal among all methods that use only information of the gradient of $f$ \cite{Nesterov2018}.
To better understand this method, which is not intuitive at first glance, many interpretations have been developed \cite{AS2022,AO2017,DO2019}.
Remarkably, approaches based on continuous-time dynamics have made significant progress over the past decade; Su, Boyd, and Cand\`{e}s~\cite{SBC2016} proposed the following ordinary differential equation (ODE) as a continuous-time model of the NAG method, called the NAG ODE:
\begin{equation}\label{intro:NAGflow}
    \ddot x(t) +\frac3t\dot x(t) + \nabla f(x(t)) = 0,
\end{equation}
which led to fruitful development in the ODE-based approach \cite{SDJS2022,WWJ2016,WilsonPhD,WRJ2021}.
Notably, Krichene, Bayen, and Bartlett~\cite{KBB2015} generalized this ODE to the mirror descent setting, proposing the \textit{accelerated mirror descent} (AMD) ODE having an $\mathcal{O}(t^{-2})$ convergence rate.

Following the spirit of the ODE-based approach, we analyze the behavior of the NAG ODE~\eqref{intro:NAGflow} applied to unbounded convex functions before looking into its discretization.
We point out that, roughly speaking, ``NAG ODE~\eqref{intro:NAGflow} for problem~\eqref{intro:primal} = AMD ODE for problem~\eqref{intro:dual}'':

\begin{thm}[informal version of \cref{thm:correspondence}]
    Let $x(t)$ be the solution of the NAG ODE~\eqref{intro:NAGflow} for \eqref{intro:primal} and $X(t)$ the solution of the AMD ODE~\cite{KBB2015} for \eqref{intro:dual}.
    Under appropriate correspondence in initial conditions and parameters, it holds that $X(t) = -\frac4t\dot x(t)$.
\end{thm}

Therefore, if problem~\eqref{intro:dual} has an optimal solution $p^\star$, then $p(t) \coloneqq -\frac{4}{t}\dot x(t)$, which is the negative velocity of $x(t)$ multiplied by an appropriate scalar, converges to $p^\star$ at $\mathcal{O}(t^{-2})$.
By integrating $p(s) \to p^\star$ from $s = 0$ to $t$, we also obtain that $-\frac8{t^2}(x(t) - x(0)) \to p^\star$ at $\mathcal{O}(t^{-2})$, which characterizes the diverging behavior of $x(t)$.

We show that these convergence rates are conserved in discretization.
We look into a class of accelerated gradient methods proposed by Ushiyama, Sato, and Matsuo~\cite{USM2023}, which turns out to include the original NAG~\eqref{intro:NAG} as a special case.
We show that with specified scalars $P_k, Q_k$, the normalized negative increment $p^{(k)} \coloneqq -P_k(x^{(k + 1)} - x^{(k)})$ and the normalized negative displacement $q^{(k)} \coloneqq -Q_k(x^{(k)} - x^{(0)})$ both converge to $p^\star$ at a rate of $\mathcal{O}(k^{-2})$, which is consistent with the continuous-time case.
These convergence rates are summarized in \cref{tab:intro}.

\begin{table}[htbp]
    \center
    \caption{Convergence results for the normalized negative increment (or velocity) $p$ and the normalized negative displacement $q$ of the trajectories of algorithms.
    These convergence rates hold only when $\dom f^\ast$ has a minimum-norm point $p^\star$.
    The constant $D \coloneqq D_f(x_0, p^\star)$ is the dual divergence between the initial point and $p^\star$.}
    \label{tab:intro}
	\begin{tabular}{lcc}\toprule
    	Method / ODE & Increment $p \to p^\star$ & Displacement $q \to p^\star$ \\\midrule\midrule
        \multirow{2}{*}{Gradient descent} & $\left\|p_k - p^\star\right\|^2 = \mathcal{O}(LDk^{-1})$ & $\left\|q_k - p^\star\right\|^2 = \mathcal{O}(LDk^{-1})$ \\
        & ($p_k \coloneqq -(x_{k + 1} - x_k)/\eta = \nabla f(x_k)$) & ($q_k \coloneqq -(x_k - x_0)/(\eta k)$) \\\midrule
        \multirow{2}{*}{NAG ODE~\eqref{intro:NAGflow}} & $\left\|p(t) - p^\star\right\|^2 = \mathcal{O}(Dt^{-2})$ & $\left\|q(t) - p^\star\right\|^2 = \mathcal{O}(Dt^{-2})$ \\
        & ($p(t) \coloneqq -4\dot x(t)/t$) & ($q(t) \coloneqq -8(x(t) - x(0))/t^2$) \\\midrule
        \multirow{2}{*}{NAG method~\eqref{intro:NAG}} & $\left\|p^{(k)} - p^\star\right\|^2 = \mathcal{O}(LDk^{-2})$ & $\left\|q^{(k)} - p^\star\right\|^2 = \mathcal{O}(LDk^{-2})$ \\
        & ($p^{(k)} \coloneqq -P_k\left(x^{(k + 1)} - x^{(k)}\right)$) & ($q^{(k)} \coloneqq -Q_k\left(x^{(k)} - x^{(0)}\right)$) \\\bottomrule
    \end{tabular}
\end{table}

Consequently, the NAG method, originally proposed to solve problem~\eqref{intro:primal}, can also be interpreted as a discretization of the AMD ODE~\cite{KBB2015} for problem~\eqref{intro:dual}.
Notably, the NAG method \textit{simultaneously} solves both problems~\eqref{intro:primal} and \eqref{intro:dual} at a rate of $\mathcal{O}(k^{-2})$:
\begin{itemize}
    \item In the usual setting where $f$ has a minimizer, $f(x^{(k)})$ is minimized at $\mathcal{O}(k^{-2})$ \cite{Nesterov1983}.
    In this case, the optimal solution of problem~\eqref{intro:dual} is $p^\star = 0$.
    Thus, the $\mathcal{O}(k^{-2})$ convergence for problem~\eqref{intro:dual} implies that $p^{(k)}$ and $q^{(k)}$, which are actually convex combinations of the gradient of $f$, decrease at this rate.
    \item On the other hand, even if $\inf f = -\infty$, the NAG method solves problem~\eqref{intro:dual} and finds the minimum-norm point $p^\star$ at $\mathcal{O}(k^{-2})$.
    Such faster convergence leads to more efficient unboundedness judgment compared to gradient descent.
    We also show that instead of $f(x)$ itself, $g(x) \coloneqq f(x) - \langle p^\star, x\rangle$ is minimized at a consistent rate with the usual case.
\end{itemize}
Such a simultaneous property would be particularly useful for \textit{possibly unbounded} objectives, where one wants to efficiently find a minimizer \textit{or} detect unboundedness with a certificate.

\subsection{Background and Related Work}\label{subsec:related}
\paragraph{Unbounded convex objectives}
As important instances of (possibly) unbounded problems, we refer to \textit{matrix scaling}~\cite{HHS2024,Idel2016,RS1989,Sinkhorn1964} and \textit{geometric programming}~\cite{BKVH2007,BLNW2020}.
In these problems, $\dom f^\ast$ is a closed set; hence, problem~\eqref{intro:dual} always has an optimal solution.
In geometric programming, $\dom f^\ast$ becomes a polytope.
Therefore, the boundedness of $f$ is equivalent to membership in the polytope~\cite{BLNW2020}, and a separating hyperplane serves as a certificate of unboundedness~\cite{HHS2024}.
In matrix scaling, the boundedness of $f$ is equivalent to the ``scalability'' of the input matrix~\cite{RS1989}.
If the input is ``unscalable'', a graph-theoretic certificate of unboundedness can be obtained from the limit of the Sinkhorn iterations~\cite{HHS2024}.

Matrix scaling and geometric programming can be generalized in a representation-theoretic way to \textit{operator scaling}~\cite{AGLOW2018,FSG2023,GGOW2020,Gurvits2004} and \textit{non-commutative optimization}~\cite{BFGOWW2019,HNW2023}.
These problems are formulated as geodesically convex optimization on Hadamard manifolds, which is a class of Riemannian manifolds that includes Euclidean spaces.

Hirai and Sakabe~\cite{HS2024} studied the asymptotic behavior of gradient descent for unbounded objectives on Hadamard manifolds and showed that the gradient norm converges to its infimum.
In the Euclidean case, this result yields convergence of the gradient $\nabla f(x_k) \to p^\star$, while their result was only qualitative.
This work extends this result by showing quantitative convergence rates.
Technically, this progress is a result of an additional assumption; the previous work shows convergence to the minimum-norm point of $\overline{\dom f^\ast}$, whereas this work assumes the existence of a minimum-norm point of $\dom f^\ast$ without taking the closure to obtain quantitative results.

\paragraph{Gradient norm minimization}
Turning to a different issue, in the bounded case, our convergence results $p^{(k)}, q^{(k)} \to p^\star = 0$ can be understood as the construction of convex combinations of the gradient of $f$ whose square norm decreases at $\mathcal{O}(k^{-2})$.
As a similar result, we refer to \cite[Theorem~2.2.4]{Nesterov2018}, which constructs a convex combination of the gradient that decreases at $\mathcal{O}(k^{-4})$ using a variant of NAG.
Also, Shi et al.~\cite{CSY2022,SDJS2022} showed that in another standard version of NAG, the minimum gradient so far achieves $o(k^{-3})$, without taking a convex combination.
However, these two results depend on the distance to a minimizer, whereas ours does not utilize information about the distance.
In the distance-free setting, Kim and Fessler~\cite{KF2021} developed the optimal gradient norm minimization method under a fixed number of iterations, where the resulting gradient norm bound coincides with our convergence rates, including the dependence on the parameters $L$ and $D$.
However, we are not sure whether this bound is still optimal if convex combinations are allowed.

\subsection{Organization of This Paper}
The rest of this paper is organized as follows.
\cref{sec:preliminaries} defines notation and summarizes fundamental properties of convex functions and duality.
\cref{sec:gd} interprets gradient descent as a gradient norm minimization algorithm.
\cref{sec:continuous-time} focuses on the NAG ODE~\eqref{intro:NAGflow} and demonstrates that it accelerates the convergence results presented in the previous section.
\cref{sec:discrete-time} deals with discretized algorithms: accelerated gradient methods, and derives convergence rates consistent with the continuous-time case.
\cref{sec:numerical} presents numerical results.
\cref{sec:conclusion} discusses possible directions for future research.

\section{Preliminaries}\label{sec:preliminaries}
Throughout this paper, $\RR$ denotes the set of real numbers, $\RR_{> 0}$ and $\RR_{\ge 0}$ denote the sets of positive and nonnegative real numbers, respectively, $[a, b]$ denotes the closed interval between $a, b \in \RR$, and $\NN$ denotes the set of nonnegative integers.
In the Euclidean space $\RR^n$, $\langle\bullet, \bullet\rangle$ denotes the Euclidean inner product, and $\|v\| \coloneqq \sqrt{\langle v, v\rangle}$ denotes the Euclidean norm of $v \in \RR^n$.
For a vector or a matrix $A$, $A^\top$ denotes its transpose.
The identity matrix is denoted by $I$.
For symmetric matrices $A$ and $B$, we write $A \preceq B$ if $B - A$ is positive semidefinite.
For a subset $S \subseteq \RR^n$, $\overline{S}$ denotes its closure and $\conv S$ denotes its convex hull.
For a function $x(t)$ parametrized by a real parameter $t$, $\dot x(t)$ denotes its derivative at $t$; if $x(t)$ is defined only in $t \in \RR_{\ge 0}$, then $\dot x(0)$ denotes the right-hand derivative.

\subsection{Unbounded Convex Functions and Duality}
This paper mainly focuses on convex functions on $\RR^n$.
For a function $f: \RR^n \to \RR \cup \{\infty\}$, its domain is defined as $\dom f \coloneqq \{x \in \RR^n \mid f(x) < \infty\}$.
A convex function $f: \RR^n \to \RR \cup \{\infty\}$ is said to be proper if $\dom f \neq \emptyset$ and finite-valued if $\dom f = \RR^n$.
The gradient of $f$ at $x \in \RR^n$ is denoted by $\nabla f(x)$, if it exists.
We also use a notation of $\nabla f(\RR^n) \coloneqq \{\nabla f(x) \mid x \in \RR^n\}$ if $f$ is finite-valued and differentiable on $\RR^n$.
For $L > 0$, we let $\mathcal{F}_L(\RR^n)$ be the set of $L$-smooth convex functions on $\RR^n$, i.e., finite-valued differentiable convex functions with $L$-Lipschitz continuous gradients.
We also use a notation of $\mathcal{F}_{< \infty}(\RR^n) \coloneqq \bigcup_{L > 0}\mathcal{F}_L(\RR^n)$.
We summarize characterizations of convexity and $L$-smoothness, which will be used in later analysis:
\begin{prop}[{e.g., \cite[Section~3.1.3]{BV2004}}]
    \label{prop:convexity}
    Let $f: \RR^n \to \RR \cup \{\infty\}$ be differentiable on $\dom f$. Then, $f$ is convex iff $\dom f$ is convex and $f(y) \ge f(x) + \langle\nabla f(x), y - x\rangle$ for all $x, y \in \dom f$.
\end{prop}
\begin{prop}[{e.g., \cite[Theorem~5.8]{Beck2017}}]
    \label{prop:smoothness}
    Let $f: \RR^n \to \RR$ be a differentiable convex function.
    The following claims are equivalent:
    \begin{enumerate}
        \item[(i)] $f$ is $L$-smooth,
        \item[(ii)] $f(y) \le f(x) + \langle\nabla f(x), y - x\rangle + \frac{L}2\|y - x\|^2$ for all $x, y \in \RR^n$,
        \item[(iii)] $f(y) \ge f(x) + \langle\nabla f(x), y - x\rangle + \frac1{2L}\|\nabla f(y) - \nabla f(x)\|^2$ for all $x, y \in \RR^n$.
    \end{enumerate}
\end{prop}
Applying (ii) with $y = x - \nabla f(x)/L$ implies $\|\nabla f(x)\|^2/(2L) \le f(x) - f(y)$, which yields the following gradient norm upper bound:
\begin{cor}
    \label{cor:smoothness}
    For $f \in \mathcal{F}_L(\RR^n)$ and $x \in \RR^n$, it holds that $\|\nabla f(x)\|^2 \le 2L(f(x) - \inf_{x' \in \RR^n}f(x'))$.
\end{cor}

Next, we introduce the Legendre--Fenchel conjugate, which is a convex function on the dual space.
For a proper convex function $f: \RR^n \to \RR \cup \{\infty\}$, its Legendre--Fenchel conjugate is defined as $f^\ast(p) \coloneqq \sup_{x \in \RR^n}\langle p, x\rangle - f(x)$.
It is known that $f^\ast$ is a proper lower-semicontinuous convex function; moreover, if $f$ is lower-semicontinuous, then $f^{\ast\ast} = f$ (e.g., \cite[Theorem~E.1.3.5]{HL2001}).
When $p = \nabla f(y)$, \cref{prop:convexity} implies that the supremum in the definition is attained at $x = y$, and thus $f^\ast(\nabla f(y)) = \langle\nabla f(y), y\rangle - f(y)$.
From this point of view, the pair $(f, f^\ast)$ induces a nonnegative divergence measure between primal and dual points, referred to as the dual divergence:
\begin{dfn}[Dual divergence, see e.g., {\cite[Appendix~A]{BMDG2005}}]
    \label{dfn:divergence}
    Let $f: \RR^n \to \RR \cup \{\infty\}$ be a proper convex function.
    The \textbf{dual divergence} associated with $f$ is defined as
    \[
        D_{f}(x, p) \coloneqq f(x) + f^\ast(p) - \langle p, x\rangle \quad (x \in \RR^n, p \in \RR^n).
    \]
\end{dfn}
\begin{rem}
    $D_{f}(x, p) < \infty$ holds if and only if $x \in \dom f$ and $p \in \dom f^\ast$.
\end{rem}
Note that when $p$ is given by $p = \nabla f(y)$, by using $f^\ast(\nabla f(y)) = \langle \nabla f(y), y\rangle - f(y)$ we obtain $D_f(x, \nabla f(y)) = f(x) - f(y) - \langle\nabla f(y), x - y\rangle$, which is a well-known formula for the \textit{Bregman divergence} originally introduced in \cite[Equation~(1.4)]{Bregman1967}.
Thus, the dual divergence can be considered a generalized version of the Bregman divergence, in the sense that $p$ does not have to be a gradient of $f$.
For further discussion about the relation between the dual and Bregman divergences, see e.g., \cite[Chapter~1]{Amari2016} or \cite[Appendix~A]{BMDG2005}.

We next see that the set of the gradients $\nabla f\left(\RR^n\right)$ is almost identical to the domain of the Legendre--Fenchel conjugate:
\begin{prop}[{\cite[Equation~(1.4)]{Hirai2024}, \cite[Equation~(3.26)]{HS2024}}]
    \label{prop:dual_space}
    For any $f \in \mathcal{F}_{<\infty}(\RR^n)$,
    \[
        \nabla f\left(\RR^n\right) \subseteq \dom f^\ast \subseteq \overline{\nabla f\left(\RR^n\right)}.
    \]
    In particular, $\overline{\dom f^\ast} = \overline{\nabla f\left(\RR^n\right)}$.
\end{prop}
Since $\overline{\dom f^\ast} = \overline{\nabla f\left(\RR^n\right)}$ is a closed convex set, it has a unique minimum-norm point.
This point plays a particularly important role in our analysis, and the following symbols are used throughout this paper.
\begin{dfn}\label{dfn:minimum-norm point}
    Let $f \in \mathcal{F}_L(\RR^n)$.
    \begin{itemize}
        \item The minimum-norm point of $\overline{\dom f^\ast}$ is denoted by $p^\star$, i.e., $p^\star \coloneqq \argmin_{p \in \overline{\dom f^\ast}}\|p\|$.
        \item A function $g \in \mathcal{F}_L(\RR^n)$ is defined as $g(x) \coloneqq f(x) - \langle p^\star, x\rangle$.
    \end{itemize}
\end{dfn}
\begin{rem}
    In the case of $\inf_{x \in \RR^n}f(x) > -\infty$, it holds that $p^\star = 0$ and $g = f$.
\end{rem}
\begin{prop}\label{prop:pg}
    Let $f \in \mathcal{F}_{<\infty}(\RR^n)$, with $p^\star$ and $g$ defined as above.
    \begin{enumerate}
        \item[(a)] For any $p \in \overline{\dom f^\ast}$, it holds that $\langle p - p^\star, p^\star\rangle \ge 0$.
        \item[(b)] For any $p \in \overline{\dom f^\ast}$, it holds that $\|p - p^\star\|^2 \le \|p\|^2 - \|p^\star\|^2$.
        \item[(c)] For any $x \in \RR^n$, it holds that $\langle\nabla g(x), p^\star\rangle \ge 0$.
        \item[(d)] For any $x \in \RR^n$ and $s \ge 0$, it holds that $g(x + sp^\star) \ge g(x)$.
        \item[(e)] $p^\star \in \dom f^\ast$ holds if and only if $\inf_{x \in \RR^n}g(x) > -\infty$.
        \item[(f)] The dual divergence $D_f(\bullet, p^\star)$ is given by $D_f(y, p^\star) = g(y) - \inf_{x \in \RR^n}g(x)$.
    \end{enumerate}
\end{prop}
\begin{proof}
    (a) is a general property of the minimum-norm point of a convex set (e.g., \cite[Theorem~A.3.1.1]{HL2001}).
    Then, (b) follows from $\|p\|^2 - \|p^\star\|^2 - \|p - p^\star\|^2 = 2\langle p - p^\star, p^\star\rangle \ge 0$.
    (c) follows from $\langle \nabla g(x), p^\star\rangle = \langle \nabla f(x) - p^\star, p^\star\rangle$ and applying (a) with $p \coloneqq \nabla f(x) \in \nabla f(\RR^n) \subseteq \dom f^\ast \subseteq \overline{\dom f^\ast}$.
    (d) is obtained by integrating (c).
    Finally, (e) and (f) follow from $f^\ast(p^\star) = -\inf_{x \in \RR^n}g(x)$ and the definitions of $\dom f^\ast, D_f$, and $g$.
\end{proof}
Note that $g$ usually has no minimizer when $p^\star \neq 0$.
This is because if $x^\star \in \RR^n$ is a minimizer, then (d) implies that every point of the form $x^\star - sp^\star (s > 0)$ is also a minimizer, which is not a typical situation.

\section{Gradient Descent as Gradient Norm Minimization}
\label{sec:gd}
Before looking into accelerated methods, we first study the behavior of gradient descent applied to possibly unbounded objectives.
In particular, we analyze the convergence properties of the gradient $\nabla f(x_k)$.
We should remark that the key idea in this analysis has already been mentioned in \cite{HS2024}, although that work did not state explicit convergence rates that can be derived from this idea.
In this section, we present a self-contained explanation of the idea and derive explicit convergence rates, together with an iteration complexity for detecting unboundedness.
These results serve as a foundation for the subsequent analysis on accelerated methods.

For $f \in \mathcal{F}_L(\RR^n)$, consider gradient descent:
\begin{equation}
    \label{eqn:gd}
    x_{k + 1} \coloneqq x_k - \eta_k\nabla f(x_k), \qquad (\eta_k > 0: \text{step size}),
\end{equation}
which is going to minimize $f$ under an appropriate choice of $\eta_k$.
It is known (e.g., \cite[Equation~(1.2.22)]{Nesterov2018}) that not only the function value of $f$, this algorithm also minimizes the gradient norm as $\|\nabla f(x_k)\|^2 = \mathcal{O}(k^{-1})$, given that $f$ is bounded from below.
Recently, Hirai and Sakabe~\cite[Theorem~3.15]{HS2024} showed that even if $f$ is lower-unbounded, it holds that $\nabla f(x_k) \to p^\star$, where $p^\star$ is the minimum-norm point of $\overline{\nabla f(\RR^n)} = \overline{\dom f^\ast}$.

Such convergence of the gradient $\nabla f(x_k)$ can be explained by a general theory of \textit{mirror descent}~\cite{NY1983}, which is formalized as follows.
For a short while, to avoid confusion, we use the symbols $E \coloneqq \RR^n$ and $E^\ast \coloneqq \RR^n$ for the primal and dual spaces, respectively.
Let $\Psi^\ast \in \mathcal{F}_{L_{\Psi^\ast}}(E^\ast)$, called the dual of \textit{distance-generating function}, $\Psi \coloneqq \Psi^{\ast\ast}$ be its Legendre--Fenchel conjugate,\footnote{In the usual convention, a strongly-convex $\Psi: \RR^n \to \RR\cup\{\infty\}$ is first defined and then consider $\Psi^\ast \in \mathcal{F}_{<\infty}(\RR^n)$ as its Legendre--Fenchel conjugate. However, in this paper, we directly define smooth convex $\Psi^\ast$. This allows us to formulate the entire discussion in the language of smooth convexity, without introducing strong convexity and subgradient and without discussing well-definedness of implicitly-defined algorithms.} and $\mathcal{X} \coloneqq \dom\Psi \subseteq E$.
Consider an \textit{objective function} $F$ as a real-valued $L_F$-smooth convex function defined on an open set including $\mathcal{X}$, so that $\nabla F$ is defined everywhere on $\mathcal{X}$; from now on, we simply consider them defined on $\mathcal{X}$, i.e., $F: \mathcal{X} \to \RR$ and $\nabla F: \mathcal{X} \to E^\ast$.
Then, the (unconstrained) mirror descent algorithm can be formulated as
\begin{equation}
    \label{eqn:mirror}
    X_k \coloneqq \nabla \Psi^\ast\left(\theta_k\right),  \qquad \theta_{k + 1} \coloneqq \theta_k - \eta_k\nabla F\left(X_k\right), \qquad (\eta_k > 0: \text{step size}).
\end{equation}
Since this formulation looks different from standard ones due to the use of $\Psi^\ast$ instead of $\Psi$, we see the equivalence to the standard ones.
For instance, under several assumptions on $\nabla \Psi$, some references (\cite[Section~4.2]{Bubeck2015}, \cite[Section~2.1.3]{WilsonPhD}) define mirror descent by
\begin{equation}
    \label{eqn:mirror2}
    \nabla\Psi(X_{k + 1}) - \nabla\Psi(X_k) = -\eta_k\nabla F(X_k).
\end{equation}
Then, by putting $\theta_k = \nabla \Psi(X_k)$, which is equivalent to $\nabla\Psi^\ast(\theta_k) = X_k$ (e.g., \cite[Corollary~E.1.4.4]{HL2001}), one can confirm that the schemes \eqref{eqn:mirror} and \eqref{eqn:mirror2} are equivalent.
Another standard formulation (e.g., \cite[Section~9]{Beck2017}) is given as Bregman proximal gradient descent, which is written as follows using the language of dual divergence in \cref{dfn:divergence}:
\[
    X_{k + 1} \coloneqq \argmin_{x \in \mathcal{X}}\left\{F(X_k) + \langle \nabla F(X_k), x - X_k\rangle +\frac1{\eta_k}D_{\Psi}(x, \nabla\Psi(X_k))\right\}.
\]
Then, the first-order optimality condition gives $\nabla F(X_k) + \frac1{\eta_k}(\nabla\Psi(X_{k + 1}) - \nabla\Psi(X_k)) = 0$, which is equivalent to \eqref{eqn:mirror2}.

\paragraph{Convergence analysis of mirror descent}
Bauschke, Bolte, and Teboulle~\cite{BBT2017} and Lu, Freund, and Nesterov~\cite{LFN2018} independently showed that algorithm~\eqref{eqn:mirror} achieves an $\mathcal{O}(k^{-1})$ convergence rate, although they did not refer to it as ``mirror descent''.
In the following, we reproduce their result in our language: using the dual divergence instead of the Bregman divergence, and without using strong convexity.
Unlike their original proofs, our proof is based on the energy function argument.
Although this proof is not the main subject of this paper, it might be of independent interest.

\begin{dfn}[Energy function for mirror descent]
    Using the trajectory $(X_k, \theta_k)_{k \in \NN}$ of mirror descent~\eqref{eqn:mirror}, an energy function with respect to a reference point $w \in \mathcal{X}$ is defined as
    \[
        V_w^{(k)} \coloneqq a_k(F(X_k) - F(w)) + D_{\Psi^\ast}(\theta_k, w) \quad (k \in \NN),
    \]
    where $a_k \coloneqq \sum_{i = 0}^{k - 1}\eta_i$ ($a_0 \coloneqq 0$) and $D_{\Psi^\ast}(\theta_k, w)$ is the dual divergence defined in \cref{dfn:divergence}.
\end{dfn}

\begin{lem}
    Under a step-size condition of $\eta_k \le \frac1{L_FL_{\Psi^\ast}}$, it holds that $V_w^{(k + 1)} \le V_w^{(k)}$.
\end{lem}
\begin{proof}
    \begin{align*}
        V_w^{(k + 1)}& - V_w^{(k)}\\
        =\ & a_{k + 1}(F(X_{k + 1}) - F(X_k)) + \eta_k(F(X_k) - F(w)) + \Psi^\ast(\theta_{k + 1}) - \Psi^\ast(\theta_k) + \langle\theta_k - \theta_{k + 1}, w\rangle\\
        \le\ &a_{k + 1}\left(\frac{L_F}2\|X_{k + 1} - X_k\|^2 + \langle\nabla F(X_k), X_{k + 1} - X_k\rangle\right) + \eta_k\langle\nabla F(X_k), X_k - w\rangle\\
        & + \Psi^\ast(\theta_{k + 1}) - \Psi^\ast(\theta_k) + \langle\theta_k - \theta_{k + 1}, w\rangle\\
        =\ &\frac{a_{k + 1}}{\eta_k}\left(\Psi^\ast(\theta_{k + 1}) - \Psi^\ast(\theta_k) +\langle\theta_k - \theta_{k + 1}, \nabla\Psi^\ast(\theta_{k + 1})\rangle\right)\\
        & + \frac{a_k}{\eta_k}\left(\Psi^\ast(\theta_k) - \Psi^\ast(\theta_{k + 1}) - \langle\theta_k - \theta_{k + 1}, \nabla\Psi^\ast(\theta_k)\rangle\right)\\
        & + \frac{a_{k + 1}L_F}2\|\nabla\Psi^\ast(\theta_{k + 1}) - \nabla\Psi^\ast(\theta_k)\|^2\\
        \le\ &-\frac12\left(\frac{a_k + a_{k + 1}}{L_{\Psi^\ast}\eta_k} - a_{k + 1}L_F\right)\|\nabla\Psi^\ast(\theta_{k + 1}) - \nabla\Psi^\ast(\theta_k)\|^2\\
        \le\ &0,
    \end{align*}
    where the first inequality follows from the $L_F$-smoothness and the convexity of $F$ (\cref{prop:convexity,prop:smoothness}~(ii)), the second equality is obtained by the algorithm~\eqref{eqn:mirror} and $a_{k + 1} = a_k + \eta_k$, the second inequality follows from the $L_{\Psi^\ast}$-smooth convexity of $\Psi^\ast$ (\cref{prop:smoothness}~(iii)), and the last inequality follows from the step-size condition.
\end{proof}
Using this nonincreasing property, it holds that
\[
    a_k(F(X_k) - F(w)) \le a_k(F(X_k) - F(w)) + D_{\Psi^\ast}(\theta_k, w) = V_w^{(k)} \le V_w^{(0)} = D_{\Psi^\ast}(\theta_0, w),
\]
from which we obtain the following convergence result:

\begin{prop}[{see also \cite[Theorem~1]{BBT2017} and \cite[Theorem~3.1]{LFN2018}}]
    \label{prop:mirror}
    Let $\Psi^\ast \in \mathcal{F}_{L_{\Psi^\ast}}(\RR^n)$, $\mathcal{X} \coloneqq \dom \Psi$, and $F:\mathcal{X} \to \RR$ be an $L_F$-smooth convex function with $\nabla F$ defined everywhere on $\mathcal{X}$.
    Let $(X_k, \theta_k)_{k \in \NN}$ be the trajectory of mirror descent~\eqref{eqn:mirror} with an initial point $\theta_0 \in \RR^n$.
    Then, under the step-size condition of $\eta_i \le \frac1{L_FL_{\Psi^\ast}}\ (\forall i \le k - 1)$, for any $w \in \mathcal{X}$, it holds that
    \[
        F(X_k) - F(w) \le \frac{D_{\Psi^\ast}(\theta_0, w)}{a_k},
    \]
    where $a_k \coloneqq \sum_{i = 0}^{k - 1}\eta_i$.
\end{prop}
\begin{cor}
    \label{cor:mirror}
    Under the setting above, the following claims hold:
    \begin{enumerate}
        \item[(a)] When $\eta_i = \frac1{L_FL_{\Psi^\ast}}$, it holds that $F(X_k) - F(w) \le \frac{L_FL_{\Psi^\ast}D_{\Psi^\ast}(\theta_0, w)}k$.
        \item[(b)] When the step size is chosen so that $a_k \to \infty$, it holds that $F(X_k) \to \inf_{x \in \mathcal{X}}F(x)$.
        \item[(c)] If $F$ has a minimizer $x^\star$ in $\mathcal{X}$, then it holds that $F(X_k) - \min_{x \in \mathcal{X}}F(x) \le \frac{D_{\Psi^\ast}(\theta_0, x^\star)}{a_k}$.
    \end{enumerate}
\end{cor}
\begin{proof}
    (a) and (c) are obvious.
    (b) is obtained by taking $w \in \mathcal{X}$ such that $F(w)$ is arbitrarily close to $\inf_{x \in \mathcal{X}}F(x)$, or arbitrarily small in the case of $\inf_{x \in \mathcal{X}}F(x) = -\infty$.
\end{proof}
\begin{rem}
    The previous studies~\cite{BBT2017, LFN2018} show that the step size can be chosen as large as $\eta_i \simeq 1/L_\mathrm{rel}$, where $L_\mathrm{rel}$ is the so-called \textit{relative smoothness parameter} of $F$ with respect to $\Psi$.
    However, we do not discuss such a relaxation, as it is not necessary for our purposes; in the specific cases discussed below, it always holds $L_\mathrm{rel} = L_FL_{\Psi^\ast}$.
\end{rem}

\paragraph{Gradient descent as mirror descent in the dual space}
Gradient descent~\eqref{eqn:gd} can be understood as special cases of mirror descent~\eqref{eqn:mirror} as follows.
By substituting the first equation of \eqref{eqn:mirror} into the second, we obtain the following recurrence formula of $(\theta_k)_{k \in \NN}$:
\[
    \theta_{k + 1} \coloneqq \theta_k - \eta_k (\nabla F)\circ(\nabla \Psi^\ast)(\theta_k),
\]
where $\circ$ denotes the function composition.
Then, one could imagine that if the composition of gradients $(\nabla F)\circ(\nabla \Psi^\ast)$ is made into a single gradient $\nabla f$, then this recurrence formula becomes gradient descent.
A standard way to do this would be to consider the \textit{Euclidean case}, where $\Psi^\ast$ is defined as $\Psi^\ast(\theta) \coloneqq \|\theta\|^2/2$ and hence $\nabla\Psi^\ast$ becomes the identity map.
In this case, the resulting problem is to minimize $F \in \mathcal{F}_{<\infty}(\RR^n)$ and the general convergence result of mirror descent (\cref{prop:mirror}) gives a convergence guarantee for $F(X_k)$.

Hirai and Sakabe~\cite{HS2024} proposed another way to obtain gradient descent from mirror descent, which is to consider the \textit{norm-minimization case}.
Instead of $\Psi^\ast$, let $F$ be the squared-norm function, i.e., $F(p) \coloneqq \|p\|^2/2$.
Then, the recurrence formula again becomes gradient descent:
\[
    \theta_{k + 1} \coloneqq \theta_k - \eta_k\nabla\Psi^\ast(\theta_k).
\]
In this case, however, the general convergence result (\cref{prop:mirror}) works differently; the objective function $F$ is now the squared-norm function, and therefore \cref{prop:mirror} guarantees that $\|X_k\|^2/2 = \|\nabla\Psi^\ast(\theta_k)\|^2/2$ is minimized in $\mathcal{X} = \dom\Psi$.
Hence, using the standard notation of $x_k = \theta_k$, $f = \Psi^\ast \in \mathcal{F}_{<\infty}(\RR^n)$, and $L = L_{\Psi^\ast}$, \cref{prop:mirror} now reads the following:
\begin{thm}
    \label{thm:gradientdescent}
    Let $f \in \mathcal{F}_L(\RR^n)$ and $(x_k)_{k \in \NN}$ be the trajectory of gradient descent~\eqref{eqn:gd} with an initial point $x_0 \in \RR^n$.
    Under the step-size condition of $\eta_i \le 1/L\ (\forall i \le k - 1)$, for any $p \in \dom f^\ast$, it holds that
    \[
        \|\nabla f(x_k)\|^2 - \|p\|^2 \le \frac{2D_f(x_0, p)}{a_k},
    \]
    where $a_k \coloneqq \sum_{i = 0}^{k - 1}\eta_i$.
\end{thm}
\begin{cor}
    \label{cor:gradientdescent}
    Under the setting of \cref{thm:gradientdescent}, the following claims hold:
    \begin{enumerate}
        \item[(a)] When $\eta_i = 1/L$, it holds that $\|\nabla f(x_k)\|^2 - \|p\|^2 \le \frac{2LD_f(x_0, p)}k$ for all $p \in \dom f^\ast$.
        \item[(b)] When the step size is chosen so that $a_k \to \infty$, then $\nabla f(x_k)$ converges to the minimum-norm point of $\overline{\dom f^\ast}$.
        \item[(c)] If $\dom f^\ast$ has a minimum-norm point $p^\star$, then it holds that $\|\nabla f(x_k) - p^\star\|^2 \le \frac{2D_f(x_0, p^\star)}{a_k}$.
        \item[(d)] In particular, if $\dom f^\ast$ has a minimum-norm point $p^\star$ and $\eta_i = 1/L$, then it holds that
        \[
            \|\nabla f(x_k) - p^\star\|^2 \le \|\nabla f(x_k)\|^2 - \|p^\star\|^2 \le \frac{2LD_f(x_0, p^\star)}k.
        \]
    \end{enumerate}
\end{cor}
\begin{proof}
    (a) is obvious.
    (b) follows from $\|\nabla f(x_k)\|^2 \to \inf_{p \in \dom f^\ast}\|p\|^2$ and the uniqueness of the minimum-norm point of $\overline{\dom f^\ast}$.
    (c) and (d) follow from $\|\nabla f(x_k) - p^\star\|^2 \le \|\nabla f(x_k)\|^2 - \|p^\star\|^2$ from \cref{prop:pg}~(b).
\end{proof}

In particular, (b) shows that the increment $(x_{k + 1} - x_k)/\eta_k = -\nabla f(x_k)$ converges to $-p^\star$.
By summing this up, we obtain that the trajectory $(x_k)_{k \in \NN}$ diverges in the direction of $-p^\star$:

\begin{thm}
    Under the setting of \cref{thm:gradientdescent}, define $q_k \coloneqq -(x_k - x_0)/a_k$ for $k \ge 1$.
    \begin{enumerate}
        \item[(a)] It holds that $q_k \in \dom f^\ast$ for all $k \ge 1$.
        \item[(b)] When the step size is chosen so that $a_k \to \infty$, then $q_k$ converges to the minimum-norm point of $\overline{\dom f^\ast}$.
        \item[(c)] If $\dom f^\ast$ has a minimum-norm point $p^\star$ and $\eta_i = 1/L$, then we have the following convergence rates:
        \begin{gather*}
            \|q_k - p^\star\|^2 \le \frac{8LD_f(x_0, p^\star)}k,\qquad \|q_k\|^2 - \|p^\star\|^2 \le 2LD_f(x_0, p^\star)\frac{C_0 + \log k}k,
        \end{gather*}
        where $C_0 \coloneqq 1 + \frac{\|\nabla f(x_0)\|^2 - \|p^\star\|^2}{2LD_f(x_0, p^\star)}$.
    \end{enumerate}
\end{thm}
\begin{proof}
    It holds that $q_k$ is a weighted average of $(\nabla f(x_i))_{0 \le i \le k - 1}$:
    \begin{equation}\label{eqn:gd-q}
        q_k = -\frac1{a_k}\sum_{i = 0}^{k - 1}(x_{i + 1} - x_i) = \frac1{a_k}\sum_{i = 0}^{k - 1}\eta_i\nabla f(x_i).
    \end{equation}
    Then, (a) follows from $\nabla f(x_i) \in \dom f^\ast$ and the convexity of $\dom f^\ast$.
    Since $\nabla f(x_i)$ converges to $p^\star$ from \cref{thm:gradientdescent}~(b), its weighted average also converges to $p^\star$ (known as the Stolz--Ces\`{a}ro theorem, e.g., \cite[Chapter~3~Theorem~1.22]{Muresan2008}), which implies (b).
    The former inequality of (c) is obtained by
    \[
        \|q_k - p^\star\|
        \le \frac1k\sum_{i = 0}^{k - 1}\|\nabla f(x_i) - p^\star\|
        \le \sqrt{2LD_f(x_0, p^\star)}\frac{1 + \sum_{i = 1}^{k - 1}\frac1{\sqrt{i}}}k
        \le \sqrt{2LD_f(x_0, p^\star)}\frac2{\sqrt{k}},
    \]
    where the first inequality follows from \eqref{eqn:gd-q} with $\eta_i = 1/L, a_k = k/L$ and the triangle inequality, the second follows from \cref{cor:gradientdescent}~(d) for $i \ge 1$ and applying \cref{cor:smoothness} to $g(x) \coloneqq f(x) - \langle p^\star, x\rangle$ for $i = 0$, and the last follows from $1 + \sum_{i = 1}^{k - 1}\frac1{\sqrt{i}} \le 2 + \int_1^{k - 1}\frac1{\sqrt{t}}\rmd t = 2\sqrt{k - 1} \le 2\sqrt{k}$.
    For the latter inequality of (c), we have
    \[
        \|q_k\|^2 - \|p^\star\|^2 \le \frac1k\sum_{i = 0}^{k - 1}\left(\|\nabla f(x_i)\|^2 - \|p^\star\|^2\right) \le 2LD_f(x_0, p^\star)\frac{C_0 + \sum_{i = 2}^{k - 1}\frac1i}k,
    \]
    where the first inequality follows from Jensen's inequality for \eqref{eqn:gd-q} with $\eta_i = 1/L, a_k = k/L$ and the last inequality follows from \cref{cor:gradientdescent}~(d) and the definition of $C_0$.
    Finally, by using $\sum_{i = 2}^{k - 1}\frac1i \le \log k$, we obtain the desired inequality.
\end{proof}

\paragraph{Detecting unboundedness}
Finally, we explain how these results can be used to detect the lower-unboundedness of $f$.
Clearly, if one detects $p^\star \neq 0$, then $f$ must be lower-unbounded.
For an explicit iteration complexity to detect this, the following property, which is satisfied by the examples in \cref{sec:geometric,sec:numerical}, would be useful:
\begin{dfn}\label{dfn:boundeddual}
    For $f \in \mathcal{F}_{<\infty}(\RR^n)$ and $M \in \RR$, we write
    \[
        f^\ast \le M
    \]
    if $f^\ast(p) \le M$ for all $p \in \dom f^\ast$.
\end{dfn}
Given this property, the dual divergence $D_f(x_0, p^\star)$ is upper bounded by a quantity that is computable at the beginning of the algorithm:
\begin{prop}\label{prop:boundeddual}
    Suppose $f^\ast \le M$ for some $M \in \RR$.
    Then, $\dom f^\ast$ has a minimum-norm point $p^\star$.
    Moreover, for any $x_0 \in \RR^n$, it holds that
    \[
        D_f(x_0, p^\star) \le M + f(x_0) + \|x_0\|\|\nabla f(x_0)\|.
    \]
\end{prop}
\begin{proof}
    Since $f^\ast$ is lower semicontinuous, for any $p \in \overline{\dom f^\ast}$, it holds that $f^\ast(p) \le M < \infty$.
    Therefore, $\dom f^\ast$ is closed; in particular, $\dom f^\ast$ has a minimum-norm point $p^\star$.
    Then, for any $x_0 \in \RR^n$,
    \[
        D_f(x_0, p^\star) = f(x_0) + f^\ast(p^\star) - \langle p^\star, x_0\rangle \le M + f(x_0) + \|x_0\|\|\nabla f(x_0)\|,
    \]
    where we used $\|p^\star\| \le \|\nabla f(x_0)\|$ for the inequality.
\end{proof}
By combining \cref{cor:gradientdescent,prop:boundeddual}, we obtain explicit convergence bounds that can be computed at the beginning of the algorithm.
Consequently, it is possible to detect unboundedness.
For instance, in the simplest case where $\eta_i = 1/L$ and $x_0 = 0$, we have the following results:
\begin{cor}\label{cor:gd-detection}
    Let $f \in \mathcal{F}_L(\RR^n)$ satisfy $f^\ast \le M$, and $(x_k)_{k \in \NN}$ be the trajectory of gradient descent~\eqref{eqn:gd} with the initial point $x_0 \coloneqq 0$ and the step size $\eta_i = 1/L$.
    \begin{enumerate}
        \item[(a)] For all $k \ge 1$, it holds that $\|\nabla f(x_k) - p^\star\|^2 \le \|\nabla f(x_k)\|^2 - \|p^\star\|^2 \le \frac{2L(M + f(0))}k$.
        \item[(b)] If $\|\nabla f(x_k)\|^2 > \frac{2L(M + f(0))}k$ for some $k$, then $f$ is lower-unbounded.
        \item[(c)] If $p^\star \neq 0$, the condition in (b) is satisfied for all $k > \frac{2L(M + f(0))}{\|p^\star\|^2}$.
    \end{enumerate}
\end{cor}
\begin{proof}
    (a) is direct from \cref{cor:gradientdescent}~(d) and \cref{prop:boundeddual}.
    Then, (b) follows by the fact that $p^\star \neq 0$ implies $f$ is lower-unbounded.
    (c) follows from $\|p^\star\| \le \|\nabla f(x_k)\|$.
\end{proof}

\section{Accelerated Gradient Descent in Continuous Time}
\label{sec:continuous-time}
In this section, we consider continuous-time ODEs that model accelerated gradient methods.
In particular, we analyze the behavior of the NAG ODE proposed by Su, Boyd, and Cand\`{e}s~\cite{SBC2016} when applied to lower-unbounded objectives.
This section is divided into two main parts corresponding to two distinct approaches to analysis.
The first part consists of \cref{subsec:AMD,subsec:correspondence}, where we consider a general framework of the accelerated mirror descent (AMD) ODE by Krichene, Bayen, and Bartlett~\cite{KBB2015}.
Then, analogously to \cref{sec:gd}, we look into the Euclidean and norm-minimization cases and discuss a correspondence between them.
The second part is \cref{subsec:continuous-analog}, where the analysis is based on a modified energy function suitable for the unbounded case.
The analysis in this part is suitable for discretization, leading to the analysis in \cref{sec:discrete-time}.
Both analyses yield the same convergence result on the velocity of the NAG ODE: $p(t) \coloneqq -\frac{r + 2}{t}\dot x(t)$ converges to the minimum-norm point $p^\star$ of $\dom f^\ast$ at $\mathcal{O}(t^{-2})$.

\subsection{Accelerated Mirror Descent ODE}
\label{subsec:AMD}
In the following, we consider a framework of the \textit{accelerated mirror descent (AMD) ODE} proposed by Krichene, Bayen, and Bartlett~\cite{KBB2015} and briefly summarize its convergence properties.
This ODE can be regarded as a continuous-time accelerated version of mirror descent and includes the NAG ODE \cite{SBC2016} as a special case.

Consider the same setting as for mirror descent in \cref{sec:gd}: let $\Psi^\ast \in \mathcal{F}_{< \infty}(\RR^n), \mathcal{X} \coloneqq \dom\Psi$ where $\Psi \coloneqq \Psi^{\ast\ast}$, and $F: \mathcal{X} \to \RR$ be an $L_F$-smooth convex function with $\nabla F$ defined everywhere on $\mathcal{X}$.
Under this setting, Krichene, Bayen, and Bartlett~\cite{KBB2015} proposed the following AMD ODE with parameter $R > 0$:
\begin{subequations}\label{eqn:AMF}
\begin{empheq}[left=\empheqlbrace]{align}
    \dot X(t) &= \frac{R}{t}\left(\nabla\Psi^\ast(Z(t)) - X(t)\right),\label{subeq:AMF1}\\
    \dot Z(t) &= -\frac{t}{R}\nabla F(X(t)),\label{subeq:AMF2}
\end{empheq}
\begin{equation}
    Z(0) \coloneqq Z_0 \in E^\ast, X(0) \coloneqq \nabla\Psi^\ast(Z_0)\ (\text{initial condition}).
\end{equation}
\end{subequations}
Note that $X(t)$ is a weighted average of gradients of $\Psi^\ast$, in the sense that \eqref{subeq:AMF1} can be written as $\frac{\rmd}{\rmd t}\left(t^RX(t)\right) = Rt^{R - 1}\nabla\Psi^\ast(Z(t))$ or equivalently $X(t) = \frac{\int_0^tw(\tau)\nabla\Psi^\ast(Z(\tau))\rmd\tau}{\int_0^tw(\tau)\rmd\tau}$ with $w(\tau) \coloneqq \tau^{R - 1}$.
Hence, by $\nabla\Psi^\ast(Z(\tau)) \in \mathcal{X}$ and the convexity of $\mathcal{X}$, it holds that $X(t) \in \mathcal{X}$ for all $t \ge 0$.
Also, it has been shown in \cite[Theorem~1]{KBB2015} that ODE~\eqref{eqn:AMF} has a unique solution in $C^1(\RR_{\ge 0})$.

\paragraph{Convergence analysis of the AMD ODE}
Krichene, Bayen, and Bartlett~\cite{KBB2015} gave an elegant convergence analysis for the AMD ODE~\eqref{eqn:AMF}, which was based on an energy function including the Bregman divergence.
In the following, we reproduce their analysis in a slightly generalized manner; we use the dual divergence instead of the Bregman divergence to cover the case where $\Psi$ is not differentiable.

\begin{dfn}[Energy function for the AMD ODE]
    \label{dfn:Lyapunov}
    Using the solution $(X, Z)$ of ODE~\eqref{eqn:AMF}, an energy function with respect to a reference point $w \in \mathcal{X}$ is defined as
    \[
        V_w(t) \coloneqq \frac{t^2}R(F(X(t)) - F(w)) + RD_{\Psi^\ast}(Z(t), w) \quad (t \ge 0),
    \]
    where $D_{\Psi^\ast}(Z(t), w)$ is the dual divergence defined in \cref{dfn:divergence}.
\end{dfn}

This energy function derives a convergence rate as follows.
By direct calculation, its time derivative is given as
\begin{equation}
    \label{eqn:amd-diff}
    \dot V_w(t) = \frac{(2 - R)t}R\bigl(F(X(t)) - F(w)\bigr) + t\bigl(F(X(t)) - F(w) + \langle\nabla F(X(t)), w - X(t)\rangle\bigr),
\end{equation}
whose latter term is nonpositive by the convexity of $F$ (\cref{prop:convexity}).
Therefore, when $R \ge 2$, it holds that $\dot V_{x^\star} \le 0$ where $x^\star$ is any minimizer of $F$ in $\mathcal{X}$.
Then for any $t > 0$, it holds that
\[
    \frac{t^2}R\left(F(X(t)) - F^\star\right) \le \frac{t^2}R\left(F(X(t)) - F^\star\right) + RD_{\psi^\ast}(Z(t), x^\star) = V_{x^\star}(t) \le V_{x^\star}(0) = RD_{\psi^\ast}(Z_0, x^\star).
\]
Thus, we obtain the following convergence rate:

\begin{prop}[{\cite[Theorem~2]{KBB2015}}]
    \label{prop:AMF_value}
    Let $\Psi^\ast \in \mathcal{F}_{<\infty}(\RR^n)$, $\mathcal{X} \coloneqq \dom \Psi$, and $F:\mathcal{X} \to \RR$ be an $L$-smooth convex function with $\nabla F$ defined everywhere on $\mathcal{X}$.
    Suppose $F^\star \coloneqq \min_{x \in \mathcal{X}}F(x)$ exists and let $x^\star \in \mathcal{X}$ be a minimizer, i.e., $F(x^\star) = F^\star$.
    Then, the solution $(X, Z)$ of ODE~\eqref{eqn:AMF} with $R \ge 2$ satisfies $F(X(t)) - F^\star \le \frac{R^2D_{\psi^\ast}(Z_0, x^\star)}{t^2}$ for all $t > 0$.
\end{prop}

\subsection{Primal and Dual Correspondence}\label{subsec:correspondence}
Analogously to \cref{sec:gd}, we look into two special cases of the AMD ODE~\eqref{eqn:AMF}: the \textit{Euclidean case} and the \textit{norm-minimization case}, and we point out that the resulting two ODEs are equivalent dynamics viewed from different perspectives.
In particular, the convergence in the norm-minimization case can be understood as convergence of ``velocity'' in the Euclidean case, which characterizes its diverging behavior for unbounded objectives.

The Euclidean case is the case where the distance-generating function is given by $\Psi^\ast(Z) \coloneqq \|Z\|^2/2$, which represents the plain minimization problem of $F \in \mathcal{F}_{< \infty}(\RR^n)$.
In this case, for clarity, we use lowercase symbols $(x, z, f, r)$ instead of $(X, Z, F, R)$.
Then, ODE~\eqref{eqn:AMF} in this case becomes the following:
\begin{subequations}\label{eqn:NAGflow}
    \begin{empheq}[left=\empheqlbrace]{align}
        \dot x(t) &= \frac{r}t\left(z(t) - x(t)\right),\label{subeq:NAGflow1}\\
        \dot z(t) &= -\frac{t}r\nabla f(x(t)),\label{subeq:NAGflow2}
    \end{empheq}
    \begin{equation}\label{subeq:NAGFinitial}
        z(0) \coloneqq x(0) \coloneqq x_0 \in \RR^n\ (\text{initial point}).
    \end{equation}
\end{subequations}
Since \eqref{subeq:NAGflow1} can be written as $z(t) = x(t) + \frac{t}r\dot x(t)$, differentiating this yields that ODE~\eqref{eqn:NAGflow} is equivalent to $\frac{\rmd}{\rmd t}\left(x(t) + \frac{t}r\dot x(t)\right) = -\frac{t}r\nabla f(x(t))$, which coincides with the one proposed by Su, Boyd, and Cand\`{e}s~\cite{SBC2016} as a continuous-time model of Nesterov's accelerated gradient (NAG) method~\cite{Nesterov1983}.
Hence, in this paper, ODE~\eqref{eqn:NAGflow} is referred to as the NAG ODE.

The norm-minimization case is the case where the objective function is given by $F(X) \coloneqq \|X\|^2/2$, which represents the norm-minimization problem over a convex set $\mathcal{X} = \dom\Psi^\ast$.
In this case, ODE~\eqref{eqn:AMF} becomes the following:
\begin{gather}
    \left\{
    \begin{aligned}
        \dot X(t) &= \frac{R}t\left(\nabla\Psi^\ast(Z(t)) - X(t)\right),\\
        \dot Z(t) &= -\frac{t}RX(t),
    \end{aligned}
    \right.\label{eqn:normflow}\\
    Z(0) \coloneqq Z_0 \in \RR^n, X(0) \coloneqq \nabla\Psi^\ast(Z_0).\notag
\end{gather}
We point out that ODEs \eqref{eqn:NAGflow} and \eqref{eqn:normflow} represent the same dynamics, under the correspondence of $\Psi^\ast = f$ and $R = r + 2$:
\begin{thm}
    \label{thm:correspondence}
    Let $f \in \mathcal{F}_{<\infty}(\RR^n)$, $r > 0$, and $x_0 \in \RR^n$ be an arbitrary point.
    Let $(x, z)$ be the $C^1(\RR_{\ge 0})$ solution of ODE~\eqref{eqn:NAGflow} with parameter $r$ and initial condition $x(0) \coloneqq z(0) \coloneqq x_0$, and let $(X, Z)$ be the $C^1(\RR_{\ge 0})$ solution of ODE~\eqref{eqn:normflow} with $\Psi^\ast = f$, $R = r + 2$ and initial condition $Z(0) \coloneqq x_0, X(0) \coloneqq \nabla f(x_0)$.
    Then, for all $t > 0$, $(x, z)$ can be expressed in terms of $(X, Z)$ as follows:
    \[
        x(t) = Z(t)\qquad
        z(t) = Z(t) - \frac{t^2}{r(r + 2)}X(t).
    \]
    Equivalently, $(X, Z)$ can be expressed in terms of $(x, z)$ as follows:
    \[
        X(t) = \frac{r(r + 2)}{t^2}\left(x(t) - z(t)\right)\qquad
        Z(t) = x(t).
    \]
\end{thm}
\begin{proof}
    We define $\tilde x(t) \coloneqq Z(t)$ and $\tilde z(t) \coloneqq Z(t) - \frac{t^2}{r(r + 2)}X(t)$.
    From the uniqueness of the solution, it suffices to show that $(\tilde x, \tilde z)$ is a $C^1(\RR_{\ge 0})$ solution to ODE~\eqref{eqn:NAGflow}.
    Since $X$ and $Z$ are $C^1(\RR_{\ge 0})$, $\tilde x$ and $\tilde z$ are also $C^1(\RR_{\ge 0})$ from their definitions.
    It holds that $\tilde x(0) = \tilde z(0) = Z(0) = x_0$ and hence the initial condition is satisfied.
    For any $t > 0$, we have
    \begin{gather*}
        \dot{\tilde x}(t) = \dot Z(t) = -\frac{t}{r + 2}X(t) = \frac{r}t(\tilde z(t) - Z(t)) = \frac{r}t(\tilde z(t) - \tilde x(t)),\\
        \dot{\tilde z}(t) = \dot Z(t) - \frac{2t}{r(r + 2)}X(t) - \frac{t^2}{r(r + 2)}\dot X(t) = -\frac{t}r\nabla f(Z(t)) = -\frac{t}r\nabla f(\tilde x(t)),
    \end{gather*}
    which shows that $(\tilde x, \tilde z)$ satisfies ODE~\eqref{eqn:NAGflow}.
    The ``equivalently'' part of the statement can be easily checked.
\end{proof}

Recall that \cref{prop:AMF_value} gives a general convergence rate of the function value $F(X(t))$.
By applying this convergence result to the norm-minimization case, and interpreting it from the viewpoint of the Euclidean case, we obtain the following convergence result regarding the velocity $\dot x(t)$ of $x(t)$:
\begin{thm}
    \label{thm:correspondence-p}
    Let $f \in \mathcal{F}_{<\infty}(\RR^n)$ and consider ODE~\eqref{eqn:NAGflow} with $r > 0$.
    Using its solution $(x, z)$, define $p(t) \coloneqq \frac{r(r + 2)}{t^2}(x(t) - z(t)) = -\frac{r + 2}t\dot x(t)$ for $t > 0$.
    \begin{enumerate}
        \item[(a)] It holds that $p(t) \in \dom f^\ast$ for all $t > 0$.
        \item[(b)] Moreover, suppose that $\dom f^\ast$ has a minimum-norm point $p^\star$.
        Then, for all $t > 0$, it holds that
        \[
            \|p(t) - p^\star\|^2 \le \|p(t)\|^2 - \|p^\star\|^2 \le \frac{2(r + 2)^2D_f(x_0, p^\star)}{t^2}.
        \]
    \end{enumerate}
\end{thm}
\begin{proof}
    Consider ODE~\eqref{eqn:normflow} with $\Psi^\ast = f, R = r + 2$, and let $(X, Z)$ be its solution with initial condition $Z(0) = x_0, X(0) = \nabla f(x_0)$.
    Then \cref{thm:correspondence} shows that $p(t) = X(t)$.
    Since $X(t) \in \mathcal{X} \coloneqq \dom f^\ast$ is a general property of the AMD ODE~\eqref{eqn:AMF}, we have (a).

    The former inequality of (b) follows from \cref{prop:pg}~(b).
    For the latter inequality, by applying \cref{prop:AMF_value} to ODE~\eqref{eqn:normflow}, we have $\|p(t)\|^2/2 - \|p^\star\|^2/2 = \|X(t)\|^2/2 - \|p^\star\|^2/2 \le \frac{R^2D_f(x_0, p^\star)}{t^2} = \frac{(r + 2)^2D_f(x_0, p^\star)}{t^2}$.
\end{proof}

The above theorem shows that if $\dom f^\ast$ has a minimum-norm point $p^\star$, then $p(t) = -\frac{r + 2}{t}\dot x$ converges to $p^\star$.
In particular, the velocity $\dot x(t)$ asymptotically points in the direction of $-p^\star$ if $p^\star \neq 0$.
By integrating this, it follows that $x(t)$ diverges in the direction of $-p^\star$.
The following theorem quantitatively characterizes the asymptotic behavior of $x(t)$:
\begin{thm}
    \label{thm:correspondence-q}
    Under the setting of \cref{thm:correspondence-p}, define $q(t) \coloneqq -\frac{2(r + 2)}{t^2}(x(t) - x(0))$ for $t > 0$.
    Then it holds that
    \begin{enumerate}
        \item[(a)] for all $t > 0$, $q(t) \in \dom f^\ast$.
    \end{enumerate}
    Moreover, suppose that $\dom f^\ast$ has a minimum-norm point $p^\star$.
    Then, we have the following convergence properties:
    \begin{enumerate}
        \item[(b)] For all $t > 0$, it holds that
        \[
            \left\|q(t) - p^\star\right\|^2 \le \frac{8(r + 2)^2D_f(x_0, p^\star)}{t^2}.
        \]
        \item[(c)] For any sufficiently large $t$, it holds that
        \[
            \left\|q(t)\right\|^2 - \|p^\star\|^2 \le 4(r + 2)^2D_f(x_0, p^\star)\frac{C + \log t}{t^2},
        \]
        for some constant $C$ independent of $t$.
    \end{enumerate}
\end{thm}
\begin{proof}
    Define $p(t)$ as in \cref{thm:correspondence-p}.
    Then it holds that
    \begin{equation}\label{eqn:q_continuous}
        q(t) = -\frac{2(r + 2)}{t^2}\int_0^t\dot x(s)\rmd s = \frac2{t^2}\int_0^tsp(s)\rmd s.
    \end{equation}
    Therefore, $q(t)$ is a weighted average of $\{p(s)\}_{s \in [0, t]}$ and thus (a) follows from \cref{thm:correspondence-p}~(a) and the convexity of $\dom f^\ast$.
    Next, (b) is obtained by
    \[
        \left\|q(t) - p^\star\right\| \le \frac2{t^2}\int_0^ts\left\|p(s) - p^\star\right\|\rmd s \le \frac2{t^2}\int_0^t(r + 2)\sqrt{2D_f(x_0, p^\star)}\rmd s = \frac{2(r + 2)\sqrt{2D_f(x_0, p^\star)}}t,
    \]
    where the first inequality follows from \eqref{eqn:q_continuous} and the triangle inequality, and the latter inequality follows from \cref{thm:correspondence-p}~(b).
    For (c), take arbitrary $t_0 > 0$.
    For any $t > t_0$, using Jensen's inequality for \eqref{eqn:q_continuous}, we have
    \begin{align*}
        \|q(t)\|^2 - \|p^\star\|^2
        &\le \frac2{t^2}\int_0^ts\left(\left\|p(s)\right\|^2 - \|p^\star\|^2\right)\rmd s\\
        &= \frac2{t^2}\int_0^{t_0}s\left(\left\|p(s)\right\|^2 - \|p^\star\|^2\right)\rmd s + \frac2{t^2}\int_{t_0}^ts\left(\left\|p(s)\right\|^2 - \|p^\star\|^2\right)\rmd s,
    \end{align*}
    and the first term of the RHS can be denoted by $C'/t^2$ with $C' \coloneqq 2\int_0^{t_0}s(\|p(s)\|^2 - \|p^\star\|^2)\rmd s$ being constant.
    Hence,
    \begin{align*}
        \|q(t)\|^2 - \|p^\star\|^2
        &\le \frac{C'}{t^2} + \frac2{t^2}\int_{t_0}^ts\left(\left\|p(s)\right\|^2 - \|p^\star\|^2\right)\rmd s\\
        &\le \frac{C'}{t^2} + \frac{4(r + 2)^2D_f(x_0, p^\star)}{t^2}\int_{t_0}^t\frac1s\rmd s\\
        &= \frac{C'}{t^2} + 4(r + 2)^2D_f(x_0, p^\star)\frac{\log(t/t_0)}{t^2},
    \end{align*}
    where we used \cref{thm:correspondence-p}~(b) for the second inequality.
\end{proof}

\begin{eg}\label{eg:tight}
    We demonstrate the tightness of the convergence rates in the above theorems by examining a one-dimensional example.
    Fix $\alpha > 0$ and consider $f: \RR \to \RR$ such that $f(x) = (x + 1)^{-\alpha} - x - 1$ for $x \ge 0$.
    For the $x < 0$ region (which does not affect the asymptotic behavior), let us say that $f$ is extended linearly, i.e.,
    \[
        f(x) = \begin{cases}
            (x + 1)^{-\alpha} - x - 1 & (x \ge 0)\\
            -(1 + \alpha)x & (x < 0)
        \end{cases},
    \]
    then, $f \in \mathcal{F}_{\alpha(\alpha + 1)}(\RR)$, $\dom f^\ast = [-(1 + \alpha), -1]$, and therefore $p^\star = -1$.
    Apply ODE~\eqref{eqn:NAGflow} with initial value $x_0 = 0$.
    Since $\nabla f(x) \ge -(1 + \alpha)$ for all $x \in \RR$, it holds that $z(t) = -\frac1r\int_0^ts\nabla f(x(s))\rmd s \le \frac{1 + \alpha}r\int_0^ts\rmd s = \frac{1 + \alpha}{2r}t^2$ for all $t > 0$.
    Then, by the comparison theorem (e.g., \cite[(16.4)~Lemma]{Amann1990}) to \eqref{subeq:NAGflow1}, $x(t)$ is no greater than the solution $\xi(t)$ to the following ODE:
    \[
        \dot \xi(t) = \frac{r}t\left(\frac{1 + \alpha}{2r}t^2 - \xi(t)\right), \qquad \xi(0) \coloneqq 0.
    \]
    One can verify that $\xi(t) = \frac{1 + \alpha}{2(r + 2)}t^2$ is a solution to it.
    By the convexity of $f$, it follows that $-\nabla f(x(t)) \ge -\nabla f(\xi(t)) = 1 + \alpha\left(1 + \frac{1 + \alpha}{2(r + 2)}t^2\right)^{-(1 + \alpha)} = 1 + \Omega(t^{-2(1 + \alpha)})$.
    Since $p(t)$ is a weighted average of $\{\nabla f(x(s))\}_{s \in [0, t]}$ (see after \eqref{eqn:AMF}), and $q(t)$ is a weighted average of $\{p(s)\}_{s \in [0, t]}$ from \eqref{eqn:q_continuous}, we have $\|p(t)\| = 1 + \Omega(t^{-2(1 + \alpha)})$ and $\|q(t)\| = 1 + \Omega(t^{-2(1 + \alpha)})$.
    Since $\alpha > 0$ can be arbitrarily small, the exponents in \cref{thm:correspondence-p}~(b) and \cref{thm:correspondence-q}~(c) cannot be improved.
\end{eg}

\begin{rem}
    Note that $\dom f^\ast$ does not always have a minimum-norm point $p^\star$.
    Even in such situations, we define $p^\star$ as the minimum-norm point of the closure $\overline{\dom f^\ast}$ (\cref{dfn:minimum-norm point}).
    Since the norm-minimization ODE~\eqref{eqn:normflow} is supposed to minimize $\|X(t)\|^2/2 = \|p(t)\|^2/2$ in $\dom f^\ast$, it is natural to expect that $\|p(t)\|^2 \to \inf_{p \in \dom f^\ast}\|p\|^2$, or equivalently, $p(t) \to p^\star$.
    However, a straightforward attempt to prove this fails due to a lack of nonpositivity of the former term of \eqref{eqn:amd-diff}.
    As already shown in \cref{sec:gd}, such a difficulty does not arise in the case of gradient/mirror descent (\cref{cor:gradientdescent}~(b)).
\end{rem}

\subsection{Primal-Based Analysis}
\label{subsec:continuous-analog}
In this section, we present another analysis of the NAG ODE~\eqref{eqn:NAGflow} applied to possibly unbounded objectives.
The analysis in this section is based on the energy function argument in the primal space.
One advantage of this approach is that the entire argument can be naturally discretized, leading to the analysis in \cref{sec:discrete-time}.
Although this approach applies only to the case of $r = 2$, it can deduce the same convergence rates for $p(t)$ and $q(t)$ as in the previous section.
We can also obtain a convergence bound for $g$ (\cref{prop:NAG_value_continuous}), which was not obtained in the previous section.
We begin our analysis with defining a generalized version of the energy function:
\begin{dfn}[Energy function for the unbounded case]
    \label{dfn:Lyapunov_continuous}
    Let $f \in \mathcal{F}_{<\infty}(\RR^n)$ and consider the solution $(x, z)$ of the NAG ODE~\eqref{eqn:NAGflow} with parameter $r = 2$.
    Using $p^\star$ and $g$ defined in \cref{dfn:minimum-norm point}, an energy function with respect to a reference point $w \in \RR^n$ is defined as
    \[
        \mathcal{V}_w(t) \coloneqq \frac{t^2}2(g(x(t)) - g(w)) + \left\|z(t) + \frac{t^2}4p^\star - w\right\|^2\quad (t \ge 0).
    \]
\end{dfn}
This energy function is a generalization of $V_w(t)$ in \cref{dfn:Lyapunov}; if $f$ is bounded from below, then it holds $p^\star = 0, g = f$, and thus $\mathcal{V}_w(t) = V_w(t)$.
However, in the case of $p^\star \neq 0$, this generalized version of the energy function provides more useful information for the divergent behavior of the NAG ODE.
The time derivative $\dot{\mathcal{V}}_w(t)$ is bounded from above as follows:
\begin{lem}\label{lem:energy_continuous}
    For any $w \in \RR^n$, the energy function $\mathcal{V}_w$ is a $C^1(\RR_{\ge 0})$ function.
    Moreover, for all $t \ge 0$, the following nonincreasing property holds:
    \[
        \dot{\mathcal{V}}_w(t) \le - \frac{t^3}4\langle\nabla g(x(t)), p^\star\rangle \le 0.
    \]
\end{lem}
\begin{proof}
    The latter inequality follows from \cref{prop:pg}~(c).
    By direct calculation, for any $t \ge 0$\footnote{For $t = 0$, care must be taken that $\dot x(0)$ is not explicitly defined from \eqref{subeq:NAGflow1}. Alternatively, one can directly obtain $\left.\frac{\rmd}{\rmd t}t^2g(x(t))\right|_{t = 0} = 0$ since $g(x(t))$ is continuous.}, we have
    \begin{align*}
        \dot{\mathcal{V}}_w(t)
        &= t\bigl(g(x(t)) - g(w)\bigr) + \frac{t^2}2\langle\nabla g(x(t)), \dot x(t)\rangle + 2\left\langle\dot z(t) + \frac{t}2p^\star, z(t) + \frac{t^2}4p^\star - w\right\rangle\\
        &= t\bigl(g(x(t)) - g(w)\bigr) + t\langle\nabla g(x(t)), z(t) - x(t)\rangle - t\left\langle\nabla g(x(t)), z(t) + \frac{t^2}4p^\star - w\right\rangle\\
        &= t\bigl(g(x(t)) - g(w) + \langle\nabla g(x(t)), w - x(t)\rangle\bigr) - \frac{t^3}4\langle\nabla g(x(t)), p^\star\rangle,
    \end{align*}
    where the second equality follows from ODE~\eqref{eqn:NAGflow} and $\nabla g = \nabla f - p^\star$.
    Since the former term of the RHS is nonpositive by the convexity of $g$, we have the desired inequality.
    Also from the above equation, $\mathcal{V}_w$ is $C^1(\RR_{\ge 0})$ since $x(t)$ is continuous.
\end{proof}

Directly from our energy function, we can evaluate the function value $g(x(t))$:
\begin{thm}\label{prop:NAG_value_continuous}
    Let $f \in \mathcal{F}_{< \infty}(\RR^n)$ and $x(t)$ be the solution to the NAG ODE~\eqref{eqn:NAGflow} with $r = 2$.
    Then, for any $w \in \RR^n$ and $t > 0$, it holds that
    \[
        g(x(t)) \le g(w) + \frac{2\|w - x_0\|^2}{t^2}.
    \]
    In particular, it holds that $\lim_{t \to \infty}g(x(t)) = \inf_{x \in \RR^n} g(x)$, including the case of $\inf_{x \in \RR^n} g(x) = -\infty$.
\end{thm}
\begin{proof}
    For all $t > 0$, it holds that
    \[
        g(x(t)) - g(w) \le \frac{2\mathcal{V}_w(t)}{t^2} \le \frac{2\mathcal{V}_w(0)}{t^2} = \frac{2\|w - x_0\|^2}{t^2},
    \]
    where the second inequality follows from \cref{lem:energy_continuous} and the other two follow from the definition of $\mathcal{V}_w$.
    The ``in particular'' part is obtained by taking $w$ such that $g(w)$ is arbitrarily close to $\inf_{x \in \RR^n}g(x)$.
\end{proof}
\begin{cor}
    If $\dom f^\ast$ has a minimum-norm point $p^\star$, then $\lim_{t \to \infty}\nabla f(x(t)) = p^\star$.
\end{cor}
\begin{proof}
    The condition $p^\star \in \dom f^\ast$ is equivalent to $\inf_{x \in \RR^n}g(x) > -\infty$ (\cref{prop:pg}~(e)).
    By $g(x(t)) \to \inf_{x \in \RR^n} g(x)$ and the smoothness of $g$ (\cref{cor:smoothness}), we have $\nabla g(x(t)) \to 0$.
\end{proof}

We now present another proof of the convergence of $p(t)$ (\cref{thm:correspondence-p}).
Although this proof applies only to the case of $r = 2$ and yields looser quantitative convergence bounds, it clarifies the role of the energy function and can be adapted to the discrete-time setting (\cref{thm:NAG_discrete}).
\begin{prop}[{Weak version of \cref{thm:correspondence-p}}]
    \label{prop:weak}
    Let $f \in \mathcal{F}_{<\infty}(\RR^n)$ and consider ODE~\eqref{eqn:NAGflow} with $r = 2$.
    Using its solution $(x, z)$, define $p(t) \coloneqq \frac{8}{t^2}(x(t) - z(t)) = -\frac{4}t\dot x(t)$ for $t > 0$.
    Then it holds that
    \begin{enumerate}
        \item[(a)] for all $t > 0$, $p(t) \in \dom f^\ast$.
    \end{enumerate}
    Moreover, suppose that $\dom f^\ast$ has a minimum-norm point $p^\star$.
    Then, we have the following convergence properties:
    \begin{enumerate}
        \item[(b)] For all $t > 0$, it holds that
        \[
            \|p(t) - p^\star\|^2 \le \frac{800D_f(x_0, p^\star)}{9t^2}.
        \]
        \item[(c)] Additionally if $p^\star \neq 0$, then for all $t > 0$, it holds that
        \[
            \|p(t)\|^2 - \|p^\star\|^2 \le \frac{944D_f(x_0, p^\star)}{9t^2} + \frac{64D_f(x_0, p^\star)^2}{\|p^\star\|^2t^4}.
        \]
    \end{enumerate}
\end{prop}
\begin{proof}
    We first rewrite $p(t)$ in terms of $\{z(s)\}_{s \in [0, t]}$.
    Since \eqref{subeq:NAGflow1} is equivalent to $\frac{\rmd}{\rmd t}\left(t^2x(t)\right) = 2tz(t)$, we have
    \begin{equation}\label{eqn:Lyapunov_continuous_proof_1}
        p(t) = \frac8{t^4}\left(-t^2z(t) + t^2x(t)\right) = \frac8{t^4}\left(-t^2z(t) + 2\int_0^tsz(s)\rmd s\right).
    \end{equation}
    Further, by using \eqref{subeq:NAGflow2}, we have
    \begin{equation}\label{eqn:Lyapunov_continuous_proof_2}
        p(t) = \frac8{t^4}\int_0^t2s\left(z(s) - z(t)\right)\rmd s = \frac8{t^4}\int_0^ts\int_s^t\tau\nabla f\left(x(\tau)\right)\rmd\tau\rmd s = \frac4{t^4}\int_0^t\tau^3\nabla f\left(x(\tau)\right)\rmd\tau,
    \end{equation}
    which implies $p(t) \in \conv\nabla f\left(\RR^n\right) \subseteq \dom f^\ast$ and hence (a).

    We next show (b).
    For all $t \ge 0$ we have
    \begin{equation}\label{eqn:Lyapunov_continuous_proof_3}
        \left\|z(t) - x_0 + \frac{t^2}4p^\star\right\|^2
        = \mathcal{V}_{x_0}(t) + \frac{t^2}2\left(g(x_0) - g(x(t))\right)
        \le \frac{t^2D_f(x_0, p^\star)}2,
    \end{equation}
    where the inequality is obtained by $g(x_0) - g(x(t)) \le D_f(x_0, p^\star)$ from \cref{prop:pg}~(f) and $\mathcal{V}_{x_0}(t) \le \mathcal{V}_{x_0}(0) = 0$ from \cref{lem:energy_continuous}.
    By combining \eqref{eqn:Lyapunov_continuous_proof_1} and \eqref{eqn:Lyapunov_continuous_proof_3}, we have
    \begin{align*}
        \|p(t) - p^\star\|
        &= \frac8{t^4}\left\|-t^2z(t) + 2\int_0^tsz(s)\rmd s - \left(\frac{t^4}4 - \int_0^t\frac{s^3}2\rmd s\right)p^\star\right\|\\
        &\le \frac8{t^4}\left(t^2\left\|z(t) - x_0 + \frac{t^2}4p^\star\right\| + 2\int_0^ts\left\|z(s) - x_0 + \frac{s^2}4p^\star\right\|\rmd s\right)\\
        &\le \frac{4\sqrt{2D_f(x_0, p^\star)}}{t^4}\left(t^3 + 2\int_0^ts^2\rmd s\right)\\
        &= \frac{20\sqrt{2D_f(x_0, p^\star)}}{3t},
    \end{align*}
    from which we obtain (b).

    In order to prove (c), we decompose $p(t)$ into the parallel and orthogonal components to $p^\star$:
    \[
        \alpha_\parallel(t) \coloneqq \frac{\langle p(t), p^\star\rangle}{\|p^\star\|^2},\qquad p_\perp(t) \coloneqq p(t) - \alpha_\parallel(t)p^\star.
    \]
    Then, $\|p_\perp(t)\|$ is bounded by $\|p_\perp(t)\| \le \|p(t) - p^\star\|$.
    On the other hand, $\alpha_\parallel(t)$ is bounded by
    \begin{align*}
        \alpha_\parallel(t) - 1
        &= \frac{4}{\|p^\star\|^2t^4}\int_0^ts^3\langle\nabla f(x(s)), p^\star\rangle\rmd s - 1\\
        &= \frac{4}{\|p^\star\|^2t^4}\int_0^ts^3\langle\nabla g(x(s)), p^\star\rangle\rmd s\\
        &\le -\frac{16}{\|p^\star\|^2t^4}\int_0^t\dot{\mathcal{V}}_{x_0}(s)\rmd s\\
        &= -\frac{16\mathcal{V}_{x_0}(t)}{\|p^\star\|^2t^4}\\
        &\le \frac{8D_f(x_0, p^\star)}{\|p^\star\|^2t^2},
    \end{align*}
    where the first equality follows from \eqref{eqn:Lyapunov_continuous_proof_2}, the second equality follows from $\nabla g = \nabla f - p^\star$, the first inequality follows from \cref{lem:energy_continuous}, the third equality follows frrom $ \mathcal{V}_{x_0}(0) = 0$, and the last inequality follows from $-2\mathcal{V}_{x_0}(t) \le t^2(g(x_0) - g(x(t))) \le t^2D_f(x_0, p^\star)$.
    Combining this and (b), we obtain
    \begin{align*}
        \|p(t)\|^2 - \|p^\star\|^2
        &= \left(\left(\alpha_\parallel(t) - 1\right)^2 + 2\left(\alpha_\parallel(t) - 1\right)\right)\|p^\star\|^2 + \|p_\perp(t)\|^2\\
        &\le \left(\frac{64D_f(x_0, p^\star)^2}{\|p^\star\|^4t^4} + \frac{16D_f(x_0, p^\star)}{\|p^\star\|^2t^2}\right)\|p^\star\|^2 + \frac{800D_f(x_0, p^\star)}{9t^2},
    \end{align*}
    from which we obtain (c).
\end{proof}

We next discuss the convergence of $q(t)$ (\cref{thm:correspondence-q}).
Interestingly, although $q(t)$ is a weighted average of $p(t)$ (\cref{eqn:q_continuous}), we can obtain a faster convergence than $p(t)$, in the sense that the constant is smaller; compare the following result with \cref{thm:correspondence-p}~(b).

\begin{thm}[Better bound for \cref{thm:correspondence-q}~(b)]
    \label{thm:NAG_continuous-q}
    Suppose that $\dom f^\ast$ has a minimum-norm point $p^\star$.
    When $r = 2$, $q(t) \coloneqq -\frac8{t^2}(x(t)- x_0)$ satisfies that for all $t > 0$,
    \[
        \|q(t) - p^\star\|^2 \le \frac{128D_f(x_0, p^\star)}{9t^2}.
    \]
\end{thm}
\begin{proof}
    For all $t > 0$, we have
    \begin{align*}
        \|q(t) - p^\star\|
        &= \frac{8}{t^4}\left\|t^2x_0 - t^2x(t) - \frac{t^4}8p^\star\right\|\\
        &= \frac{8}{t^4}\left\|2\int_0^ts\left(x_0 - z(s)\right) - \frac{s^3}4p^\star\rmd s\right\|\\
        &\le \frac{16}{t^4}\int_0^ts\left\|x_0 - z(s) - \frac{s^2}4p^\star\right\|\rmd s\\
        &\le \frac{8\sqrt{2D_f(x_0, p^\star)}}{t^4}\int_0^ts^2\rmd s\\
        &= \frac{8\sqrt{2D_f(x_0, p^\star)}}{3t},
    \end{align*}
    where we used $\frac{\rmd}{\rmd t}\left(t^2x(t)\right) = 2tz(t)$ for the second equality and \eqref{eqn:Lyapunov_continuous_proof_3} for the second inequality.
\end{proof}

\section{Accelerated Gradient Method for Unbounded Objectives}\label{sec:discrete-time}
The goal of this section is to establish a discrete-time analogue of the theory in \cref{sec:continuous-time}.
We look into a discretized accelerated gradient method proposed by Ushiyama, Sato, and Matsuo~\cite{USM2023}, which turns out to be a generalization of the original Nesterov's accelerated gradient method~\cite{Nesterov1983}.
We discretize the convergence analysis in \cref{subsec:continuous-analog} to obtain convergence rates consistent with the continuous-time ones.
We also explain some concrete convergence results for geometric programming.

\subsection{Accelerated Gradient Method}
We consider the following algorithm, which can be obtained as the explicit case of the framework proposed in~\cite[Theorem~5.4]{USM2023}:

\begin{dfn}[accelerated gradient method~\cite{USM2023}]
We consider the \textbf{accelerated gradient method} for $f \in \mathcal{F}_L(\RR^n)$ defined by the following update scheme:
\begin{subequations}\label{eqn:NAG-method}
    \begin{empheq}[left=\empheqlbrace]{align}
        y^{(k)} - x^{(k)} &= \frac{\delta^+A_k}{A_{k + 1}}\left(z^{(k)} - x^{(k)}\right),\label{subeq:NAGmethod1}\\
        z^{(k + 1)} - z^{(k)} &= -\frac{\delta^+A_k}4\nabla f\left(y^{(k)}\right),\label{subeq:NAGmethod2}\\
        x^{(k + 1)} - x^{(k)} &= \frac{\delta^+A_k}{A_k}\left(z^{(k + 1)} - x^{(k + 1)}\right)\label{subeq:NAGmethod3},
    \end{empheq}
    \begin{equation}
        x^{(0)} \coloneqq z^{(0)} \coloneqq x_0\ (\text{initial point}),
    \end{equation}
\end{subequations}
where $\delta^+A_k \coloneqq A_{k + 1} - A_k$ and $(A_k)_{k \in \NN}$ is an increasing parameter satisfying the following condition:
\begin{equation}\label{eqn:NAG_stepsize}
    A_0 = 0,\qquad 0 < \delta^+A_k \le 2\sqrt{\frac{A_{k + 1}}L}\quad (\forall k \in \NN),\qquad \delta^+A_k = \Theta(k).
\end{equation}
\end{dfn}

Let us briefly explain how this scheme can be considered as a discretization of the NAG ODE~\eqref{eqn:NAGflow} with $r = 2$.
Regard $A_k$ as a time factor corresponding to $t^2$ in ODE~\eqref{eqn:NAGflow}.
Then, \eqref{subeq:NAGmethod2} and \eqref{subeq:NAGmethod3} respectively correspond to \eqref{subeq:NAGflow2} and \eqref{subeq:NAGflow1}, except that $y^{(k)}$ is used instead of $x^{(k)}$.
The use of $y^{(k)}$ seems quite technical, but this trick is needed to show a nonincreasing property of the energy function, which will be defined later.
It is known that this method achieves a consistent rate of $f(x^{(k)}) - f^\star = \mathcal{O}(1/A_k)$ if $f^\star \coloneqq \min_{x \in \RR^n}f(x)$ exists~\cite[Theorem~5.4]{USM2023}; we will reproduce this convergence rate in a generalized manner in \cref{thm:NAG_value_discrete}.

We now derive another formulation of \eqref{eqn:NAG-method} by eliminating $z^{(k)}$.
It holds that
\begin{equation}\label{eqn:another-formulation_1}
    y^{(k + 1)} - x^{(k + 1)} = \frac{\delta^+A_{k + 1}}{A_{k + 2}}\left(z^{(k + 1)} - x^{(k + 1)}\right) = \frac{A_k\delta^+A_{k + 1}}{A_{k + 2}\delta^+A_k}\left(x^{(k + 1)} - x^{(k)}\right),
\end{equation}
where the former equality is \eqref{subeq:NAGmethod1} and the latter is \eqref{subeq:NAGmethod3}.
On the other hand, one can see that \eqref{subeq:NAGmethod1} is equivalent to
\begin{equation*}
    (\delta^+A_k)z^{(k)} = A_{k + 1}y^{(k)} - A_kx^{(k)},
\end{equation*}
and \eqref{subeq:NAGmethod3} is equivalent to
\begin{equation}\label{eqn:equivalent-third}
    (\delta^+A_k)z^{(k + 1)} = A_{k + 1}x^{(k + 1)} - A_kx^{(k)},
\end{equation}
and by eliminating $z^{(k)}, z^{(k + 1)}$ in \eqref{subeq:NAGmethod2} using the above two equations, we obtain
\begin{equation}\label{eqn:another-formulation_2}
    A_{k + 1}\left(x^{(k + 1)} - y^{(k)}\right) = -\frac{(\delta^+A_k)^2}4\nabla f\left(y^{(k)}\right).
\end{equation}
Thus, \eqref{eqn:another-formulation_1} and \eqref{eqn:another-formulation_2} yield the following self-contained update formulae:
\begin{subequations}\label{eqn:NAG-method-2}
    \begin{empheq}[left=\empheqlbrace]{align}
        x^{(k + 1)} &\coloneqq y^{(k)} - \frac{(\delta^+A_k)^2}{4A_{k + 1}}\nabla f\left(y^{(k)}\right),\label{subeq:NAG21}\\
        y^{(k + 1)} &\coloneqq x^{(k + 1)} + \frac{A_k\delta^+A_{k + 1}}{A_{k + 2}\delta^+A_k}\left(x^{(k + 1)} - x^{(k)}\right)\label{subeq:NAG22}
    \end{empheq}
    \begin{equation}
        x^{(0)} \coloneqq y^{(0)} \coloneqq x_0\ (\text{initial point}).
    \end{equation}
\end{subequations}

We now point out that the reduced algorithm~\eqref{eqn:NAG-method-2} coincides with the original accelerated gradient method proposed by Nesterov~\cite{Nesterov1983} when $(A_k)_{k \in \NN}$ is chosen to be as large as possible.
Indeed, if one recursively chooses $A_{k + 1}$ ($k = 0, 1, \ldots$) so that the latter inequality in \eqref{eqn:NAG_stepsize} becomes tight, it holds that
\begin{equation}\label{eqn:Nesterov-1}
    \frac{(\delta^+A_k)^2}{4A_{k + 1}} = \frac1L,
\end{equation}
and therefore
\[
    L(\delta^+A_{k + 1})^2 = 4A_{k + 2} = 4\delta^+A_{k + 1} + 4A_{k + 1} = 4\delta^+A_{k + 1} + L(\delta^+A_k)^2 \quad (k \in \NN),
\]
which implies
\[
    \frac{L}2\delta^+A_{k + 1} = 1 + \sqrt{1 + \left(\frac{L}2\delta^+A_k\right)^2} \quad (k \in \NN).
\]
One can also check that $\delta^+A_0 = \frac4L$ holds under this choice.
Therefore, we have the following recursive formula for $\alpha_k \coloneqq \frac{L}4\delta^+A_k$ ($k \in \NN$):
\begin{equation}\label{eqn:Nesterov-2}
    \alpha_0 \coloneqq 1,\qquad \alpha_{k + 1} \coloneqq \frac{1 + \sqrt{1 + 4\alpha_k^2}}2 \quad (k \in \NN).
\end{equation}
Then, we have
\begin{equation}\label{eqn:Nesterov-3}
    \frac{A_k\delta^+A_{k + 1}}{A_{k + 2}\delta^+A_k} = \frac{(A_{k + 1} - \delta^+A_k)\delta^+A_{k + 1}}{A_{k + 2}\delta^+A_k} = \frac{(\alpha_k^2 - \alpha_k)\alpha_{k + 1}}{\alpha_{k + 1}^2\alpha_k} = \frac{\alpha_k - 1}{\alpha_{k + 1}} \quad (k \in \NN),
\end{equation}
where we used \eqref{eqn:Nesterov-1} and $\alpha_k \coloneqq \frac{L}4\delta^+A_k$.
Combined with \eqref{eqn:Nesterov-1}--\eqref{eqn:Nesterov-3}, algorithm~\eqref{eqn:NAG-method-2} is nothing but the original accelerated method~\eqref{intro:NAG} proposed by Nesterov~\cite{Nesterov1983}.

Another efficient choice of $(A_k)_{k \in \NN}$ is $A_k \coloneqq k(k + 1)/L$.
This choice enables us to use simple coefficients when computing $p^{(k)}$ and $q^{(k)}$, which are defined in \cref{thm:NAG_discrete,thm:NAG_discrete-q}; see \cref{rem:discrete_specific_coefficients} for the exact formulae.
Combined with $A_k \coloneqq k(k + 1)/L$, algorithm~\eqref{eqn:NAG-method-2} is now written as follows:
\begin{equation}\label{eqn:NAG-method-fixed-parameter}
    \left\{\begin{aligned}
        x^{(k + 1)} &\coloneqq y^{(k)} - \frac{k + 1}{(k + 2)L}\nabla f\left(y^{(k)}\right),\\
        y^{(k + 1)} &\coloneqq x^{(k + 1)} + \frac{k}{k + 3}\left(x^{(k + 1)} - x^{(k)}\right),
    \end{aligned}\right.\\
\end{equation}
\[
    x^{(0)} \coloneqq y^{(0)} \coloneqq x_0\ (\text{initial point}).
\]
While the latter coefficient $\frac{k}{k + 3}$ can be commonly found in literature (e.g., \cite{CSY2022,SDJS2022,SBC2016}), the former coefficient here is restricted to be $\frac{k + 1}{(k + 2)L}$ instead of an arbitrary $\eta \le 1/L$.
This restriction cannot be removed in our analysis for a technical reason.
See \cref{sec:conclusion} for more discussion.

\subsection{Convergence Results for Unbounded Objective Functions}\label{subsec:discrete}
In this section, we analyze the behavior of the accelerated gradient method~\eqref{eqn:NAG-method} (or equivalently \eqref{eqn:NAG-method-2}) applied to possibly unbounded objectives.
Our analysis is essentially a discrete-time version of the one in \cref{subsec:continuous-analog}, and we obtain convergence rates consistent with the continuous-time ones.
Some proofs of the theorems in this section are long and similar to the continuous-time ones; therefore, they are deferred to \cref{sec:proof}.
Our technical tool is again an energy function, which can be considered a discrete-time version of \cref{dfn:Lyapunov_continuous} and also a generalized version of the one in \cite{USM2023}:
\begin{dfn}[Energy function, discrete time]
    \label{dfn:Lyapunov_discrete}
    Let $f \in \mathcal{F}_L(\RR^n)$ and consider the trajectory $\left(x^{(k)}, z^{(k)}\right)$ of the accelerated gradient method~\eqref{eqn:NAG-method}.
    Using $p^\star$ and $g$ defined in \cref{dfn:minimum-norm point}, an energy function with respect to a reference point $w \in \RR^n$ is defined as
    \[
        \mathcal{V}_w^{(k)} \coloneqq \frac{A_k}2\left(g\left(x^{(k)}\right) - g(w)\right) + \left\|z^{(k)} + \frac{A_k}4p^\star - w\right\|^2 \qquad (k \in \NN).
    \]
\end{dfn}
\begin{lem}\label{lem:Lyapunov_discrete}
    Under the step-size condition~\eqref{eqn:NAG_stepsize}, for any $w \in \RR^n$, the above energy function $\mathcal{V}_w$ has the following nonincreasing property:
    \[
        \mathcal{V}_w^{(k + 1)} - \mathcal{V}_w^{(k)} \le -\frac{A_k\delta^+A_k}8\left\langle\nabla g\left(y^{(k)}\right), p^\star\right\rangle \le 0 \quad (\forall k \in \NN).
    \]
\end{lem}
\begin{proof}
The second inequality follows from \cref{prop:pg}~(c).
Our proof of the first inequality is long and is deferred to \cref{subsec:proof-lem}.
\end{proof}
\begin{thm}\label{thm:NAG_value_discrete}
    Let $f \in \mathcal{F}_L(\RR^n)$, and $\left(x^{(k)}\right)_{k \in \NN}$ be the trajectory of the accelerated gradient method \eqref{eqn:NAG-method-2} (or equivalently \eqref{eqn:NAG-method}) with an initial point $x_0 \in \RR^n$.
    Under the step-size condition~\eqref{eqn:NAG_stepsize}, for any $w \in \RR^n$ and $k \ge 1$, it holds that
    \[
        g\left(x^{(k)}\right) \le g(w) + \frac{2\|w - x_0\|^2}{A_k},
    \]
    In particular, it holds that $\lim_{k \to \infty}g\left(x^{(k)}\right) = \inf_{x \in \RR^n} g(x)$, including the case of $\inf_{x \in \RR^n} g(x) = -\infty$.
\end{thm}
\begin{proof}
    For all $k \in \NN$, it holds
    \[
        g\left(x^{(k)}\right) - g(w) \le \frac{2\mathcal{V}_w^{(k)}}{A_k} \le \frac{2\mathcal{V}_w^{(0)}}{A_k} = \frac{2\|w - x_0\|^2}{A_k},
    \]
    where the second inequality follows from \cref{lem:Lyapunov_discrete} and the other two follow from the definition of $\mathcal{V}_w$.
    The ``in particular'' part is obtained by taking $w$ such that $g(w)$ is arbitrarily close to $\inf_{x \in \RR^n}g(x)$.
\end{proof}
\begin{cor}
    If $\dom f^\ast$ has a minimum-norm point $p^\star$, then $\lim_{k \to \infty}\nabla f\left(x^{(k)}\right) = p^\star$.
\end{cor}
\begin{proof}
    The condition $p^\star \in \dom f^\ast$ is equivalent to $\inf_{x \in \RR^n}g(x) > -\infty$ (\cref{prop:pg}~(e)).
    By $g\left(x^{(k)}\right) \to \inf_{x \in \RR^n}g(x)$ and the $L$-smoothness of $g$ (\cref{cor:smoothness}), we have $\nabla g\left(x^{(k)}\right) \to 0$.
\end{proof}

We now state our main results, which are the discrete-time versions of \cref{thm:correspondence-p,thm:correspondence-q,thm:NAG_continuous-q,prop:weak}:
\begin{thm}
    \label{thm:NAG_discrete}
    Let $f \in \mathcal{F}_L(\RR^n)$, and $\left(x^{(k)}\right)_{k \in \NN}$ be the trajectory of the accelerated gradient method~\eqref{eqn:NAG-method-2} (or equivalently \eqref{eqn:NAG-method}) with an initial point $x_0 \in \RR^n$.
    Define $p^{(k)} \coloneqq -P_k\left(x^{(k + 1)} - x^{(k)}\right)$ where $P_k \coloneqq \frac{4A_kA_{k + 1}}{(\delta^+A_k)\sum_{i = 1}^{k}A_i\delta^+A_i}$ for $k \ge 1$.
    Then it holds that
    \begin{itemize}
        \item[(a)] for all $k \ge 1$, $p^{(k)} \in \dom f^\ast$.
    \end{itemize}
    Moreover, suppose that $\dom f^\ast$ has a minimum-norm point $p^\star$.
    Then, under the step-size condition~\eqref{eqn:NAG_stepsize}, we have the following convergence properties:
    \begin{itemize}
        \item[(b)] For all $k \ge 1$, it holds that 
        \[
            \left\|p^{(k)} - p^\star\right\|^2 \le B_kD_f(x_0, p^\star) = \mathcal{O}(k^{-2}),
        \]
        where $B_k \coloneqq 8\left(\frac{A_k\sqrt{A_{k + 1}} + \sum_{i = 1}^k\sqrt{A_i}\delta^+A_{i - 1}}{\sum_{i = 1}^kA_i\delta^+A_i}\right)^2$.
        \item[(c)] Additionally if $p^\star \neq 0$, then for all $k \ge 1$, it holds that
        \[
            \left\|p^{(k)}\right\|^2 - \|p^\star\|^2 \le C'_kD_f(x_0, p^\star) + \frac{C_k^2D_f(x_0, p^\star)^2}{\|p^\star\|^2} = \mathcal{O}(k^{-2}),
        \]
        where $C_k \coloneqq \frac{4A_{k + 1}}{\sum_{i = 1}^kA_i\delta^+A_i}$ and $C'_k \coloneqq B_k + 2C_k$.
    \end{itemize}
\end{thm}
\begin{proof}
    In \cref{subsec:proof-p}.
\end{proof}

\begin{thm}
    \label{thm:NAG_discrete-q}
    Under the setting of \cref{thm:NAG_discrete}, define $q^{(k)} \coloneqq -Q_k\left(x^{(k)} - x_0\right)$ where $Q_k \coloneqq \frac{4A_k}{\sum_{i = 1}^kA_i\delta^+A_{i - 1}}$ for $k \ge 1$.
    Then it holds that
    \begin{enumerate}
        \item[(a)] for all $k \ge 1$, $q^{(k)} \in \dom f^\ast$.
    \end{enumerate}
    Moreover, suppose that $\dom f^\ast$ has a minimum-norm point $p^\star$.
    Then, we have the following convergence properties:
    \begin{enumerate}
        \item[(b)] For all $k \ge 1$, it holds that 
        \[
            \left\|q^{(k)} - p^\star\right\|^2 \le \tilde B_kD_f(x_0, p^\star) = \mathcal{O}(k^{-2}),
        \]
        where $\tilde B_k \coloneqq 8\left(\frac{\sum_{i = 1}^k\sqrt{A_i}\delta^+A_{i - 1}}{\sum_{i = 1}^kA_i\delta^+A_{i - 1}}\right)^2$.
        \item[(c)] Additionally if $p^\star \neq 0$, then for all $k \ge 1$, it holds that
        \[
            \left\|q^{(k)}\right\|^2 - \|p^\star\|^2 \le \tilde C'_kD_f(x_0, p^\star) + \frac{\tilde C_k^2D_f(x_0, p^\star)^2}{\|p^\star\|^2} = \mathcal{O}(k^{-2}\log k),
        \]
        where $\tilde C_k \coloneqq \frac{A_k}{\sum_{i = 1}^kA_i\delta^+A_{i - 1}}\left(\frac{A_1\langle\nabla f(x_0) - p^\star, p^\star\rangle}{D_f(x_0, p^\star)} + 4\sum_{i = 1}^{k - 1}\frac{\delta^+A_i}{A_i}\right)$ and $\tilde C'_k \coloneqq \tilde B_k + 2\tilde C_k$.
    \end{enumerate}
\end{thm}
\begin{proof}
    In \cref{subsec:proof-q}.
\end{proof}
When the objective function satisfies $f^\ast \le M$ (\cref{dfn:boundeddual}), the following easy-to-use convergence bounds and sufficient conditions for unboundedness can be proved analogously to \cref{cor:gd-detection}:
\begin{cor}\label{cor:NAG-detection}
    Under the setting of the above theorems, suppose further that $f^\ast \le M$ and $x_0 \coloneqq 0$.
    \begin{enumerate}
        \item[(a)] For all $k \ge 1$, it holds that $\left\|q^{(k)} - p^\star\right\|^2 \le \tilde B_k(M + f(0))$.
        \item[(b)] If $\left\|q^{(k)}\right\|^2 > \tilde B_k(M + f(0))$ for some $k$, then $f$ is lower-unbounded.
        \item[(c)] If $p^\star \neq 0$, the condition in (b) is satisfied for all $k$ satisfying $\tilde B_k < \frac{\|p^\star\|^2}{M +f(0)}$.
    \end{enumerate}
    The same claims also hold with $(p^{(k)}, B_k)$ in place of $(q^{(k)}, \tilde B_k)$.
\end{cor}
Although it is not clear at first glance, \cref{cor:NAG-detection}~(c) shows, up to a constant, a quadratic speedup compared to gradient descent in detecting unboundedness (cf. \cref{cor:gd-detection}~(c)); see the following remark for a concrete bound on $\tilde B_k$.

\begin{rem}\label{rem:discrete_specific_coefficients}
    Let us observe that taking $A_k \coloneqq k(k + 1)/L$ simplifies the computations of $p^{(k)}$ and $q^{(k)}$.
    Under this choice, one can obtain
    \[
        P_k = \frac{12L}{3k + 5}, \qquad Q_k = \frac{24L}{(k + 2)(3k + 1)}.
    \]
    These simple expressions are useful in practice for computing $p^{(k)}$ and $q^{(k)}$.
    Other constants are given or bounded by
    \begin{align*}
        B_k &\le \frac{800L}{(3k + 5)^2},& C_k &= \frac{24L}{k(3k + 5)},& C'_k &\le \frac{944L}{3k(3k + 5)},\\
        \tilde B_k &\le \frac{128L}{(3k + 1)^2},& \tilde C_k &\le \frac{48L(c + 1 + \log k)}{(k + 2)(3k + 1)},& \tilde C'_k &\le \frac{96L(c + 2 + \log k)}{(k + 2)(3k + 1)},
    \end{align*}
    where $c \coloneqq \frac{\langle\nabla g(x_0), p^\star\rangle}{4LD_f(x_0, p^\star)}$.
    One can observe that these coefficients have similarities to those in the continuous-time versions: \cref{prop:weak,thm:NAG_continuous-q}.
    See \cref{app:coefficients} for detailed calculations.
\end{rem}

\begin{rem}
    Note that the accelerated gradient method~\eqref{eqn:NAG-method} does not find $x \in \RR^n$ such that $\|\nabla f(x) - p^\star\|$ is small enough.
    When such $x$ is needed, one can first use our method to find $p \in \dom f^\ast$ such that $\|p - p^\star\|$ is small enough, and then apply gradient norm minimization to the shifted objective $f_p(x) \coloneqq f(x) - \langle p, x\rangle$.
    For instance, the OGM-G algorithm by Kim and Fessler~\cite{KF2021} finds $x_N \in \RR^n$ such that $\|\nabla f_p(x_N)\| \le \frac{2\sqrt{LD_f(x_0, p)}}{N + 1}$ with $N$ times of gradient evaluations, assuming that $f_p$ has a minimizer.
    Particularly if $f^\ast \le M$, then from \cref{cor:NAG-detection}~(a) and $D_f\left(0, q^{(N)}\right) \le M + f(0)$, one can find $q^{(N)} \in \dom f^\ast$ and $x_{N} \in \RR^n$ such that
    \begin{align*}
        \left\|\nabla f(x_N) - p^\star\right\|
        &\le \left\|q^{(N)} - p^\star\right\| + \left\|\nabla f(x_N) - q^{(N)}\right\|\\
        &\le \sqrt{M + f(0)}\left(\sqrt{\tilde B_N} + \frac{2\sqrt{L}}{N + 1}\right) = \mathcal{O}(N^{-1}),
    \end{align*}
    with $2N$ times of gradient evaluations in total.
    (The same thing holds with $(p^{(N)}, B_N)$ in place of $(q^{(N)}, \tilde B_N)$.)
\end{rem}

\subsection{Geometric Programming}\label{sec:geometric}
We present some results on a specific class of convex optimization, (unconstrained) \textit{geometric programming}.
This problem is known as a generalization of matrix scaling~\cite{HHS2024,Sinkhorn1964}, and has recently attracted attention as the torus case of non-commutative optimization~\cite{BFGOWW2019}.
We refer to \cite{BKVH2007} and \cite[Chapter~4.5]{BV2004} for general introductions to geometric programming, and \cite{BLNW2020} for recent developments in interior-point methods for this problem from modern perspectives.

Geometric programming asks to minimize a function $f$ of the form
\begin{equation}\label{eqn:geometric}
    f(x) \coloneqq \log\sum_{\ell = 1}^Nc_\ell\rme^{\langle \omega_\ell, x\rangle} \qquad (x \in \RR^n),
\end{equation}
where $c_1, \ldots, c_N \in \RR_{>0}$ and $\omega_1, \ldots, \omega_N \in \RR^n$.
We define a set $\Omega \coloneqq \{\omega_1, \ldots, \omega_N\}$.
The following facts are well-known.
\begin{prop}[{see \cite{BLNW2020}}]\label{prop:geometric}
    Let $f$ be a function of the form \eqref{eqn:geometric}.
    \begin{itemize}
        \item[(a)] $f$ is an $L_\Omega$-smooth convex function, where $L_\Omega \coloneqq \max_{\omega \in \Omega}\|\omega\|^2$.
        \item[(b)] The domain of the Legendre--Fenchel conjugate is given by $\dom f^\ast = \conv\Omega$ (known as the Newton polytope).
        \item[(c)] It holds that $f^\ast \le -\log c_{\mathrm{min}}$, where $c_{\mathrm{min}} \coloneqq \min_{1 \le l \le N}c_l$.
    \end{itemize}
\end{prop}

It is direct from (b) that $p^\star$ in the geometric programming case is the minimum-norm point of the Newton polytope $\conv\Omega$, and that $f$ is bounded from below if and only if $p^\star = 0$.
From \cref{cor:NAG-detection}, $p^\star$ can be estimated by
\[
    \left\|p^{(k)} - p^\star\right\|^2 \le B_k(f(x_0) -\log c_{\mathrm{min}}),\qquad
    \left\|q^{(k)} - p^\star\right\|^2 \le \tilde B_k(f(x_0) -\log c_{\mathrm{min}}).
\]
Since $p^\star$ is a minimum-norm point of a polytope, $p^\star$ itself can be estimated more efficiently via completely different approaches: for instance, variants of the Frank--Wolfe algorithm~\cite{LJ2015} or a projected gradient method~\cite{NNG2019}\footnote{Note that the \textit{orthogonal} projection is not needed in this projected method. The orthogonal projection is nothing but the minimum-norm point problem itself!}.
However, our advantage lies in the fact that minimization of $f$, or of $g$ in the unbounded case, can be performed simultaneously, which is explained below in detail.

\cref{thm:NAG_value_discrete} provides an upper bound on the value $g(x^{(k)})$, but the bound consists of two terms; one is the function value $g(w)$ at the reference point $w$, and the other depends on the distance $\|w - x_0\|$ to it.
Generally, if $g$ does not have a minimizer, there is a trade-off between these two terms; a smaller value of $g(w)$ requires a larger distance $\|w - x_0\|$.
In the geometric programming case, B\"{u}rgisser, Li, Nieuwboer, and Walter~\cite{BLNW2020} proved a quantitative bound for this trade-off:
\begin{prop}[{\cite[Theorem~2.4]{BLNW2020}}]\label{prop:geometric-diameter}
    Let $f$ be a function of the form \eqref{eqn:geometric}, and let $\beta \coloneqq \sum_{\ell = 1}^Nc_\ell/c_{\mathrm{\min}}$.
    For any $0 < \delta < 2\beta$ and $p \in \conv\Omega$, there exists $w \in W$ such that
    \[
        \|w\| \le \frac{m}\varphi\log\left(\frac{2\beta}\delta\right)
    \]
    and
    \[
        f_p(w) \le \inf_{x \in \RR^n}f_p(x) + \delta,
    \]
    where $f_p(x) \coloneqq f(x) - \langle p, x\rangle$, $W \subseteq \RR^n$ is the span of the vectors $\omega_\ell - p$, $m \in \NN$ is the dimension of the affine hull of $\Omega$, and $\varphi > 0$ is the smallest distance from any $\omega \in \Omega$ to the affine span of any facet of $\conv\Omega$ not containing $\omega$.
\end{prop}
Note that $g$ is a special case of $f_p$ in the above proposition.
Therefore, combined with \cref{thm:NAG_value_discrete}, we obtain the following convergence rate for $g(x^{(k)})$:
\begin{thm}\label{thm:geometric-g}
    Let $f$ be a function of the form \eqref{eqn:geometric}, and apply the accelerated gradient method \eqref{eqn:NAG-method-2} (or \eqref{eqn:NAG-method}) with the initial point $x_0 \coloneqq 0$.
    Then, for all sufficiently large $k$ such that $A_k > m^2/(\beta\varphi^2)$, it holds that
    \[
        \frac{\left\|\nabla f\left(x^{(k)}\right) - p^\star\right\|^2}{2L_\Omega} \le g\left(x^{(k)}\right) - \inf_{x \in \RR^n}g(x) \le \frac{2m^2}{\varphi^2A_k}\left(1 + \log^2\left(\frac{\beta\varphi^2A_k}{m^2}\right)\right) = \mathcal{O}(k^{-2}\log^2k).
    \]
\end{thm}
\begin{proof}
    The first inequality follows from $\nabla g = \nabla f - p^\star$ and the $L_\Omega$-smoothness of $g$ (\cref{cor:smoothness}).
    \cref{thm:NAG_value_discrete} combined with \cref{prop:geometric-diameter} implies that $g\left(x^{(k)}\right) - \inf_{x \in \RR^n}g(x) \le \delta + \frac2{A_k}\left(\frac{m}{\varphi}\log\left(\frac{2\beta}\delta\right)\right)^2$ holds for all $0 < \delta < 2\beta$.
    By choosing $\delta = \frac{2m^2}{\varphi^2A_k}$, we obtain the desired inequality.
\end{proof}

\subsection{Proofs}\label{sec:proof}
In this section, we show complete proofs deferred in \cref{subsec:discrete}.

\subsubsection{Proof of \texorpdfstring{\cref{lem:Lyapunov_discrete}}{Lemma \ref{lem:Lyapunov_discrete}}}\label{subsec:proof-lem}
In this proof, we use a notation $\delta^+\phi(k) \coloneqq \phi(k + 1) - \phi(k)$ for a formula $\phi$ containing $k$.

The first step of this proof would be the most nontrivial part of the proof; we let $s_k \coloneqq \frac{(\delta^+A_k)^2}{4A_{k + 1}}$ and use \cref{prop:pg}~(d) to obtain
\begin{align}
    \delta^+\mathcal{V}_w^{(k)}
    &= \frac12\left(A_{k + 1}g\left(x^{(k + 1)}\right) - A_kg\left(x^{(k)}\right) - (\delta^+A_k)g(w)\right) + \delta^+\left\|z^{(k)} + \frac{A_k}4p^\star - w\right\|^2\notag\\
    &\le \frac12\left(A_{k + 1}g\left(x^{(k + 1)} + s_kp^\star\right) - A_kg\left(x^{(k)}\right) - (\delta^+A_k)g(w)\right) + \delta^+\left\|z^{(k)} + \frac{A_k}4p^\star - w\right\|^2.\label{eqn:proof_1}
\end{align}
The rest of the proof is a standard application of $L$-smooth convexity.
The former term of the RHS of \eqref{eqn:proof_1} is bounded by
\begin{align}
    A_{k + 1}g&\left(x^{(k + 1)} + s_kp^\star\right) - A_kg\left(x^{(k)}\right) - (\delta^+A_k)g(w)\notag\\
    \le\ & A_{k + 1}\left(g\left(y^{(k)}\right) + \left\langle\nabla g\left(y^{(k)}\right), x^{(k + 1)} + s_kp^\star - y^{(k)}\right\rangle + \frac{L}2\left\|x^{(k + 1)} + s_kp^\star - y^{(k)}\right\|^2\right)\notag\\
    &\ - A_kg\left(x^{(k)}\right) - (\delta^+A_k)g(w)\notag\\
    =\ & \left\langle\nabla g\left(y^{(k)}\right), A_{k + 1}x^{(k + 1)} - A_kx^{(k)} - \left(\delta^+A_k\right)w + A_{k + 1}s_kp^\star\right\rangle\notag\\
    &\ + A_k\left(g\left(y^{(k)}\right) - g\left(x^{(k)}\right) + \left\langle\nabla g\left(y^{(k)}\right), x^{(k)} - y^{(k)}\right\rangle\right)\notag\\
    &\ + (\delta^+A_k)\left(g\left(y^{(k)}\right) - g(w) + \left\langle\nabla g\left(y^{(k)}\right), w - y^{(k)}\right\rangle\right)\notag\\
    &\ + \frac{LA_{k + 1}}2\left\|x^{(k + 1)} + s_kp^\star - y^{(k)}\right\|^2\notag\\
    \le\ & \left\langle\nabla g\left(y^{(k)}\right), A_{k + 1}x^{(k + 1)} - A_kx^{(k)} - \left(\delta^+A_k\right)w + A_{k + 1}s_kp^\star\right\rangle\notag\\
    &\ + \frac{LA_{k + 1}}2\left\|x^{(k + 1)} + s_kp^\star - y^{(k)}\right\|^2\notag\\
    =\ & (\delta^+A_k)\left\langle\nabla g\left(y^{(k)}\right), z^{(k + 1)} - w + \frac{\delta^+A_k}4p^\star\right\rangle + \frac{L(\delta^+A_k)^4}{32A_{k + 1}}\left\|\nabla g\left(y^{(k)}\right)\right\|^2,\label{eqn:proof_2}
\end{align}
where the first inequality follows from the $L$-smoothness of $g$ (\cref{prop:smoothness}~(ii)), the second inequality follows from the convexity of $g$ (\cref{prop:convexity}), and the last equality follows from \eqref{eqn:equivalent-third}, \eqref{subeq:NAG21}, $\nabla g = \nabla f - p^\star$, and the definition of $s_k$.
On the other hand, the latter term of the RHS of \eqref{eqn:proof_1} can be converted by
\begin{align}
    &\left\|z^{(k + 1)} + \frac{A_{k + 1}}4p^\star - w\right\|^2 - \left\|z^{(k)} + \frac{A_k}4p^\star - w\right\|^2\notag\\
    &\qquad= \left\langle\delta^+z^{(k)} + \frac{\delta^+A_k}4p^\star, 2z^{(k + 1)} - \delta^+z^{(k)} + \frac{2A_{k + 1} - (\delta^+A_k)}4p^\star - 2w\right\rangle\notag\\
    &\qquad= -\frac{\delta^+A_k}2\left\langle\nabla g\left(y^{(k)}\right), z^{(k + 1)} + \frac{A_{k + 1}}4p^\star - w\right\rangle - \frac{(\delta^+A_k)^2}{16}\left\|\nabla g\left(y^{(k)}\right)\right\|^2,\label{eqn:proof_4}
\end{align}
where the last equality follows from \eqref{subeq:NAGmethod2} and $\nabla g = \nabla f - p^\star$.
By combining \eqref{eqn:proof_1}--\eqref{eqn:proof_4}, we obtain
\begin{align*}
    \delta^+\mathcal{V}_w^{(k)} \le -\frac{A_k(\delta^+A_k)}8\left\langle\nabla g\left(y^{(k)}\right), p^\star\right\rangle - \frac{(\delta^+A_k)^2}{16}\left(1 - \frac{L(\delta^+A_k)^2}{4A_{k + 1}}\right)\left\|\nabla g\left(y^{(k)}\right)\right\|^2,
\end{align*}
and the step-size condition~\eqref{eqn:NAG_stepsize} implies $\left(1 - \frac{L(\delta^+A_k)^2}{4A_{k + 1}}\right) \ge 0$, which completes the proof of the first inequality of \cref{lem:Lyapunov_discrete}.

\subsubsection{Proof of \texorpdfstring{\cref{thm:NAG_discrete}}{Theorem \ref{thm:NAG_discrete}}}\label{subsec:proof-p}
    The proof proceeds analogously to the proof of \cref{prop:weak}.

    We first rewrite $p^{(k)}$ in terms of $(z^{(i)})_{i \le k + 1}$:
    \begin{equation}
        p^{(k)}
        = \frac{4A_{k + 1}}{\sum_{i = 1}^kA_i\delta^+A_i}\left(x^{(k + 1)} - z^{(k + 1)}\right)
        = \frac4{\sum_{i = 1}^kA_i\delta^+A_i}\left(-A_kz^{(k + 1)} +\sum_{i = 1}^k(\delta^+A_{i - 1})z^{(i)}\right),\label{eqn:Lyapunov_discrete_proof_1}
    \end{equation}
    where the former equality follows from \eqref{subeq:NAGmethod3} and the latter follows by recursively applying \eqref{eqn:equivalent-third}.
    Further, by using \eqref{subeq:NAGmethod2}, we have
    \begin{align}
        p^{(k)}
        &= \frac4{\sum_{i = 1}^kA_i\delta^+A_i}\sum_{i = 1}^k(\delta^+A_{i - 1})\left(z^{(i)} - z^{(k + 1)}\right)\notag\\
        &= \frac1{\sum_{i = 1}^kA_i\delta^+A_i}\sum_{i = 1}^k(\delta^+A_{i - 1})\sum_{j = i}^k(\delta^+A_j)\nabla f\left(y^{(j)}\right)\notag\\
        &= \frac1{\sum_{i = 1}^kA_i\delta^+A_i}\sum_{j = 1}^k(\delta^+A_j)\nabla f\left(y^{(j)}\right)\sum_{i = 1}^j(\delta^+A_{i - 1})\notag\\
        &= \frac1{\sum_{i = 1}^kA_i\delta^+A_i}\sum_{j = 1}^kA_j(\delta^+A_j)\nabla f\left(y^{(j)}\right),\label{eqn:Lyapunov_discrete_proof_2}
    \end{align}
    which implies $p^{(k)} \in \conv \nabla f\left(\RR^n\right) \subseteq \dom f^\ast$ and hence (a).

    We next show (b).
    For all $k \in \NN$ we have
    \begin{equation}\label{eqn:Lyapunov_discrete_proof_3}
        \left\|z^{(k)} - x_0 + \frac{A_k}4p^\star\right\|^2
        = \mathcal{V}_{x_0}^{(k)} + \frac{A_k}2\left(g(x_0) - g\left(x^{(k)}\right)\right)
        \le \frac{A_kD_f(x_0, p^\star)}2,
    \end{equation}
    where the inequality is obtained by $g(x_0) - g(x^{(k)}) \le D_f(x_0, p^\star)$ from \cref{prop:pg}~(f) and $\mathcal{V}_{x_0}^{(k)} \le \mathcal{V}_{x_0}^{(0)} = 0$ from \cref{lem:Lyapunov_discrete}.
    By combining \eqref{eqn:Lyapunov_discrete_proof_1} and \eqref{eqn:Lyapunov_discrete_proof_3}, we have
    \begin{align*}
        \left\|p^{(k)} - p^\star\right\|
        &= \frac{4}{\sum_{i = 1}^kA_i\delta^+A_i}\left\|-A_kz^{(k + 1)} +\sum_{i = 1}^k(\delta^+A_{i - 1})z^{(i)} - \frac{A_kA_{k + 1} - \sum_{i = 1}^kA_i\delta^+A_{i - 1}}4p^\star\right\|\\
        &\le \frac{4}{\sum_{i = 1}^kA_i\delta^+A_i}\left(A_k\left\|z^{(k + 1)} - x_0 + \frac{A_{k + 1}}4p^\star\right\| + \sum_{i = 1}^k(\delta^+A_{i - 1})\left\|z^{(i)} - x_0 + \frac{A_i}4p^\star\right\|\right)\\
        &\le \frac{2\sqrt{2D_f(x_0, p^\star)}}{\sum_{i = 1}^kA_i\delta^+A_i}\left(A_k\sqrt{A_{k + 1}} + \sum_{i = 1}^k(\delta^+A_{i - 1})\sqrt{A_i}\right),
    \end{align*}
    where we used an identity $\sum_{i = 1}^kA_i\delta^+A_i + \sum_{i = 1}^kA_i\delta^+A_{i - 1} = A_kA_{k + 1} - A_0A_1$ for the first equality.
    Thus, (b) has been proved.

    In order to prove (c), we decompose $p^{(k)}$ into the parallel and orthogonal components to $p^\star$:
    \begin{equation}\label{eqn:Lyapunov_discrete_proof_4}
        \alpha_\parallel^{(k)} \coloneqq \frac{\left\langle p^{(k)}, p^\star\right\rangle}{\|p^\star\|^2},\qquad p_\perp^{(k)} \coloneqq p^{(k)} - \alpha_\parallel^{(k)}p^\star,
    \end{equation}
    then $\|p_\perp^{(k)}\|$ is bounded by $\|p_\perp^{(k)}\| \le \|p^{(k)} - p^\star\|$.
    On the other hand, $\alpha_\parallel^{(k)}$ is bounded by
    \begin{align}
        \alpha_\parallel^{(k)} - 1
        &= \frac1{\|p^\star\|^2\sum_{i = 1}^kA_i\delta^+A_i}\sum_{j = 1}^kA_j(\delta^+A_j)\left\langle\nabla f\left(y^{(j)}\right), p^\star\right\rangle - 1\notag\\
        &= \frac1{\|p^\star\|^2\sum_{i = 1}^kA_i\delta^+A_i}\sum_{j = 1}^kA_j(\delta^+A_j)\left\langle\nabla g\left(y^{(j)}\right), p^\star\right\rangle\notag\\
        &\le -\frac8{\|p^\star\|^2\sum_{i = 1}^kA_i\delta^+A_i}\sum_{j = 0}^k\left(\mathcal{V}_{x_0}^{(j + 1)} - \mathcal{V}_{x_0}^{(j)}\right)\notag\\
        &= -\frac{8\mathcal{V}_{x_0}^{(k + 1)}}{\|p^\star\|^2\sum_{i = 1}^kA_i\delta^+A_i}\notag\\
        &\le \frac{4A_{k + 1}D_f(x_0, p^\star)}{\|p^\star\|^2\sum_{i = 1}^kA_i\delta^+A_i}\label{eqn:Lyapunov_discrete_proof_5}\\
        &\eqqcolon \frac{C_kD_f(x_0, p^\star)}{\|p^\star\|^2}\notag
    \end{align}
    where the first equality follows from \eqref{eqn:Lyapunov_discrete_proof_2}, the second equality follows from $\nabla g = \nabla f - p^\star$, the first inequality follows from \cref{lem:Lyapunov_discrete} and $A_0 = 0$, the third equality follows from $\mathcal{V}_{x_0}^{(0)} = 0$, and the second inequality follows from $-2\mathcal{V}_{x_0}^{(k + 1)} \le A_{k + 1}\left(g(x_0) - g(x^{(k + 1)})\right) \le A_{k + 1}D_f(x_0, p^\star)$.
    Combining this and (b), we obtain
    \begin{align*}
        \left\|p^{(k)}\right\|^2 - \|p^\star\|^2
        &= \left(\left(\alpha_\parallel^{(k)} - 1\right)^2 + 2\left(\alpha_\parallel^{(k)} - 1\right)\right)\|p^\star\|^2 + \left\|p_\perp^{(k)}\right\|^2\\
        &\le \left(\frac{C_k^2D_f(x_0, p^\star)^2}{\|p^\star\|^4} + \frac{2C_kD_f(x_0, p^\star)}{\|p^\star\|^2}\right)\|p^\star\|^2 + B_kD_f(x_0, p^\star),
    \end{align*}
    from which we obtain (c).

    Finally, the $\mathcal{O}$ notations in the statement can be obtained by the assumption $\delta^+A_k = \Theta(k)$ in \eqref{eqn:NAG_stepsize} and recursively applying $\sum_{i = 0}^kc_k = \Theta(k^{\alpha + 1})$ for $c_k = \Theta(k^\alpha)$ ($\alpha \ge 0$).

\subsubsection{Proof of \texorpdfstring{\cref{thm:NAG_discrete-q}}{Theorem \ref{thm:NAG_discrete-q}}}\label{subsec:proof-q}
    We first show (b).
    By recursively applying \eqref{eqn:equivalent-third} and then using \eqref{eqn:Lyapunov_discrete_proof_3}, we obtain
    \begin{align*}
        \|q^{(k)} - p^\star\|
        &= \frac{4}{\sum_{i = 1}^kA_i\delta^+A_{i - 1}}\left\|A_kx_0 - \sum_{i = 1}^{k}(\delta^+A_{i - 1})z^{(i)} - \frac{\sum_{i = 1}^kA_i\delta^+A_{i - 1}}4p^\star\right\|,\\
        &\le \frac{4}{\sum_{i = 1}^kA_i\delta^+A_{i - 1}}\sum_{i = 1}^k(\delta^+A_{i - 1})\left\|x_0 - z^{(i)} - \frac{A_i}4p^\star\right\|,\\
        &\le \frac{2\sqrt{2D_f(x_0, p^\star)}}{\sum_{i = 1}^kA_i\delta^+A_{i - 1}}\sum_{i = 1}^k(\delta^+A_{i - 1})\sqrt{A_i},
    \end{align*}
    which implies (b).

    Let $p^{(k)}$ be as defined in \cref{thm:NAG_discrete}, then $q^{(k)}$ can be expressed as
    \begin{align}
        q^{(k)}
        &= -\frac{4A_k}{\sum_{i = 1}^kA_i\delta^+A_{i - 1}}\left(x^{(1)} - x_0 + \sum_{j = 1}^{k - 1}\left(x^{(j + 1)} - x^{(j)}\right)\right)\notag\\
        &= \frac{A_k}{\sum_{i = 1}^kA_i\delta^+A_{i - 1}}\left(A_1\nabla f(x_0) + \sum_{j = 1}^{k - 1}\frac{(\delta^+A_j)\sum_{i = 1}^jA_i\delta^+A_i}{A_jA_{j + 1}}p^{(j)}\right)\notag\\
        &= \frac{A_k}{\sum_{i = 1}^kA_i\delta^+A_{i - 1}}\left(A_1\nabla f(x_0) + \sum_{j = 1}^{k - 1}\left(\frac{\sum_{i = 1}^{j + 1}A_i\delta^+A_{i - 1}}{A_{j + 1}} - \frac{\sum_{i = 1}^jA_i\delta^+A_{i - 1}}{A_j}\right)p^{(j)}\right)\label{eqn:proof-q},
    \end{align}
    where the second inequality follows from $x^{(1)} = x_0 - \frac{A_1}4\nabla f(x_0)$ from \eqref{subeq:NAG21} and the definition of $p^{(j)}$, and the last can be confirmed by $A_j\sum_{i = 1}^{j + 1}A_i\delta^+A_{i - 1} - A_{j + 1}\sum_{i = 1}^jA_i\delta^+A_{i - 1} = (\delta^+A_j)(A_jA_{j + 1} - \sum_{i = 1}^jA_i\delta^+A_{i - 1}) = (\delta^+A_j)(\sum_{i = 1}^jA_i\delta^+A_i - A_0A_1)$.
    Therefore, by telescoping the sum of the coefficients, it follows that $q^{(k)}$ is a convex combination of $\nabla f(x_0)$ and $\{p^{(j)}\}_{1 \le j \le k - 1}$, which implies (a).

    In order to prove (c), we decompose $q^{(k)}$ into the parallel and orthogonal components to $p^\star$:
    \[
        \beta_\parallel^{(k)} \coloneqq \frac{\left\langle q^{(k)}, p^\star\right\rangle}{\|p^\star\|^2},\qquad q_\perp^{(k)} \coloneqq q^{(k)} - \beta_\parallel^{(k)}p^\star,
    \]
    then $\|q_\perp^{(k)}\|$ is bounded by $\|q_\perp^{(k)}\| \le \|q^{(k)} - p^\star\|$.
    On the other hand, using $\alpha_\parallel^{(j)}$ defined in \eqref{eqn:Lyapunov_discrete_proof_4}, $\beta_\parallel^{(k)}$ is bounded by
    \begin{align*}
        \beta_\parallel^{(k)} - 1
        &= \frac{A_k}{\sum_{i = 1}^kA_i\delta^+A_{i - 1}}\left(\frac{A_1\langle\nabla f(x_0), p^\star\rangle}{\|p^\star\|^2} + \sum_{j = 1}^{k - 1}\frac{(\delta^+A_j)\sum_{i = 1}^jA_i\delta^+A_i}{A_jA_{j + 1}}\frac{\left\langle p^{(j)}, p^\star\right\rangle}{\|p^\star\|^2}\right) - 1\\
        &= \frac{A_k}{\sum_{i = 1}^kA_i\delta^+A_{i - 1}}\left(\frac{A_1\langle\nabla g(x_0), p^\star\rangle}{\|p^\star\|^2} + \sum_{j = 1}^{k - 1}\frac{(\delta^+A_j)\sum_{i = 1}^jA_i\delta^+A_i}{A_jA_{j + 1}}\left(\alpha_\parallel^{(j)} - 1\right)\right)\\
        &\le \frac{A_kD_f(x_0, p^\star)}{\|p^\star\|^2\sum_{i = 1}^kA_i\delta^+A_{i - 1}}\left(\frac{A_1\langle\nabla g(x_0), p^\star\rangle}{D_f(x_0, p^\star)} + 4\sum_{j = 1}^{k - 1}\frac{\delta^+A_j}{A_j}\right)\\
        &\eqqcolon \frac{\tilde C_kD_f(x_0, p^\star)}{\|p^\star\|^2},
    \end{align*}
    where the first equality follows from \eqref{eqn:proof-q}, the second equality follows from $\nabla g = \nabla f - p^\star$ and the definition of $\alpha_\parallel^{(j)}$, and the inequality follows from \eqref{eqn:Lyapunov_discrete_proof_5}.
    Combining this with (b), we obtain
    \begin{align*}
        \left\|q^{(k)}\right\|^2 - \|p^\star\|^2
        &= \left(\left(\beta_\parallel^{(k)} - 1\right)^2 + 2\left(\beta_\parallel^{(k)} - 1\right)\right)\|p^\star\|^2 + \left\|q_\perp^{(k)}\right\|^2\\
        &\le \left(\frac{\tilde C_k^2D_f(x_0, p^\star)^2}{\|p^\star\|^4} + \frac{2\tilde C_kD_f(x_0, p^\star)}{\|p^\star\|^2}\right)\|p^\star\|^2 + \tilde B_kD_f(x_0, p^\star),
    \end{align*}
    from which we obtain (c).

    Finally, the $\mathcal{O}$ notations in the statement can be obtained by the assumption $\delta^+A_k = \Theta(k)$ and recursively applying $\sum_{i = 0}^kc_k = \Theta(\log k)$ for $c_k = \Theta(k^{-1})$ and $\sum_{i = 0}^kc_k = \Theta(k^{\alpha + 1})$ for $c_k = \Theta(k^\alpha)$ ($\alpha \ge 0$).

\section{Numerical Results}\label{sec:numerical}
In this section, we show numerical results of the accelerated gradient method~\eqref{eqn:NAG-method-fixed-parameter} applied to lower-unbounded objectives.
In the following, we consider two lower-unbounded problems.
Both problems are two-dimensional, so that the trajectories can be illustrated.
Together with the convergence of $p^{(k)}$, $q^{(k)}$ (defined in \cref{thm:NAG_discrete,thm:NAG_discrete-q}, see also \cref{rem:discrete_specific_coefficients}), and $g(x^{(k)})$ (defined in \cref{dfn:minimum-norm point}), we also display the behavior of $\nabla f(y^{(k)})$.
This is motivated by a fact that in different variants of NAG, if $f$ has a minimizer, then $\|\nabla f(y^{(k)})\|^2$ decreases at a rate of $\mathcal{O}(k^{-3})$ (\cite{CSY2022}, \cite[Theorem~2.2.6]{Nesterov2018}, \cite[Theorem~6]{SDJS2022}).

The programs used for the numerical experiments were implemented in C++ and run on a laptop computer, and the data were plotted by gnuplot.
Among the two experiments, the first one was conducted in the 128-bit floating-point format and the second in 64-bit due to a difference in convergence speeds.
The code used for the numerical experiments is available at \href{https://github.com/bekasa001/unbounded_NAG_public}{\nolinkurl{https://github.com/bekasa001/unbounded_NAG_public}}.

\paragraph{Geometric Programming}
We show a numerical result for $f:\RR^n \to \RR$ given by \eqref{eqn:geometric} with parameters $n = 2$, $N = 4$, $c = (1, 1, 1, 1)$, and $\Omega = \{(3, 0), (0, 1), (1, 2), (3, 3)\}$.
Then, the minimum-norm point of $\dom f^\ast = \conv\Omega$ is $p^\star = (0.3, 0.9)$ and $\inf g$ is given as $\inf g = \log \left(3^{0.2} + 3^{-1.8}\right)$.
We set the initial point as $x_0 = (0, 0)$.

\begin{figure}[htb]
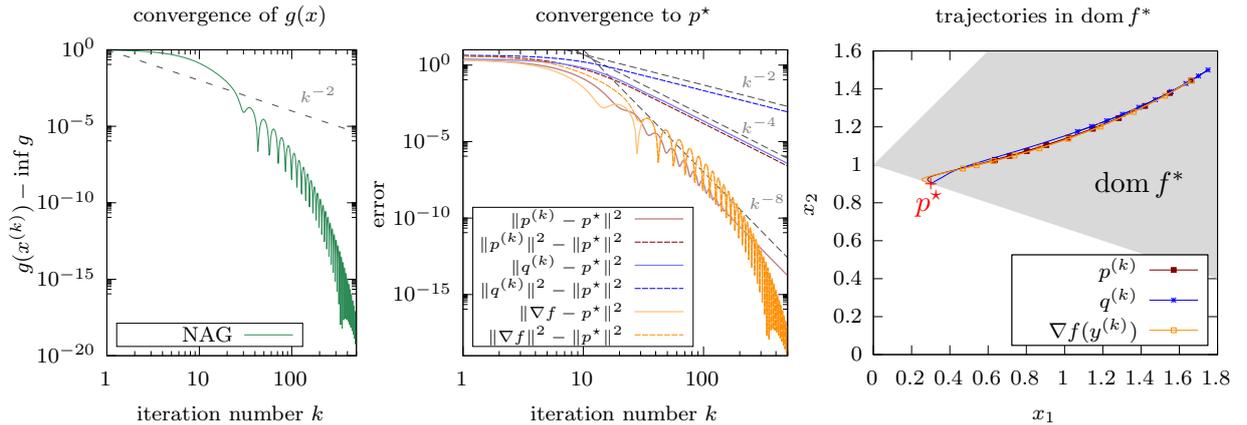

    \begin{minipage}{0.29\linewidth}
        \centering
        {\scriptsize\input{geometric_programming_g}}
    \end{minipage}
    \begin{minipage}{0.35\linewidth}
        \centering
        {\scriptsize\input{geometric_programming_dual}}
    \end{minipage}
    \begin{minipage}{0.35\linewidth}
        \centering
        {\scriptsize\input{geometric_programming_trajectory}}
    \end{minipage}
    \caption{\label{fig:geometric} Convergence behavior of accelerated gradient method~\eqref{eqn:NAG-method-fixed-parameter} applied to unbounded geometric programming.
    The leftmost plot shows the history of $g(x^{(k)}) - \inf_{x \in \RR^n}g(x)$, together with a referential line of $k^{-2}$.
    The central plot shows the convergence behavior of $p^{(k)}$, $q^{(k)}$, and $\nabla f(y^{(k)})$ using two measures: $\|\bullet -~p^\star\|^2$ and $\|\bullet\|^2 - \|p^\star\|^2$, together with referential lines of $k^{-2}$, $k^{-4}$, and $k^{-8}$.
    $\nabla f$ in the legend denotes $\nabla f(y^{(k)})$.
    The rightmost plot illustrates the trajectories of $p^{(k)}$, $q^{(k)}$, and $\nabla f(y^{(k)})$; the first ten points are marked to show their speeds of approach.
    }
\end{figure}

Convergence behavior of $p^{(k)}$, $q^{(k)}$, $\nabla f(y^{(k)})$, and $g(x^{(k)})$ is shown in \cref{fig:geometric}.
The convergence of $g(x^{(k)})$ appears to be super-polynomial, contrary to the theoretical guarantee of $\mathcal{O}(k^{-2}\log^2k)$ (\cref{thm:geometric-g}).
The convergence of $\nabla f(y^{(k)})$ also appears to be super-polynomial, whereas we do not have any theoretical results even for the quantitative property $\nabla f(y^{(k)}) \to p^\star$.
Compared to this, the convergence speeds of $p^{(k)}$ and $q^{(k)}$ are slower.
This relation could be explained by the fact that $p^{(k)}$ and $q^{(k)}$ are convex combinations of $\{\nabla f(y^{(j)})\}_{0 \le j \le k}$ (\cref{eqn:Lyapunov_discrete_proof_2,eqn:proof-q}); even if $\nabla f(y^{(k)})$ converges rapidly, $p^{(k)}$ and $q^{(k)}$ must contain errors in previous $\nabla f(y^{(j)})$s.
Still, their observed convergence speeds are much faster than our theoretical guarantees (\cref{thm:NAG_discrete,thm:NAG_discrete-q}) except for $\|q^{(k)}\|^2 - \|p^\star\|^2$.

\paragraph{Projection onto an Ellipsoid}
We consider objective functions of the form
\begin{equation}\label{eqn:ellipsoid}
    f(x) \coloneqq \sqrt{1 + \langle x, Ax\rangle} + \langle b, x\rangle \qquad (x \in \RR^n),
\end{equation}
where $b \in \RR^n$ and $A$ is an $n \times n$ positive definite symmetric matrix with the operator norm $\|A\|$.
Then, one can show\footnote{These facts can be proved as follows. Note that $f$ in \eqref{eqn:ellipsoid} is obtained by an affine transformation from $f_0(x) \coloneqq \sqrt{1 + \|x\|^2}$ as $f(x) = f_0(A^{1/2}x) + \langle b, x\rangle$. Then, $f \in \mathcal{F}_{\|A\|}(\RR^n)$ follows from $f_0 \in \mathcal{F}_1(\RR^n)$ \cite[Example~5.14]{Beck2017}. It is known that $\dom f_0^\ast(p)$ is the closed unit ball and $f_0^\ast(p) = -\sqrt{1 - \|p\|^2}$ \cite[Section~4.4.14]{Beck2017}. The Legendre--Fenchel conjugate of an affine transformation is given by \cite[Theorem~12.3]{Rockafellar1970}, in particular, $f^\ast(p) = f_0^\ast(A^{-1/2}(p - b)) = -\sqrt{1 - \langle p - b, A^{-1}(p - b)\rangle} \le 0$ and $\dom f^\ast = \{p + b \in \RR^n \mid \langle p, A^{-1}p\rangle \le 1\}$.} that $f \in \mathcal{F}_{\|A\|}(\RR^n)$, $f^\ast \le 0$, and $\dom f^\ast$ equals an ellipsoid $\{p + b \in \RR^n \mid \langle p, A^{-1}p\rangle \le 1\}$.
We conducted a numerical experiment on such $f$ with parameters $n = 2$, $A = \diag(8, 2)$, and $b = (3, 3)$.
Then, the minimum-norm point of $\dom f^\ast = \{p + b \mid \langle p, A^{-1}p\rangle \le 1\}$ is $p^\star = (1, 2)$ and $\inf g = 0$.
We set the initial point as $x_0 = (0, 0)$.

\begin{figure}[htb]
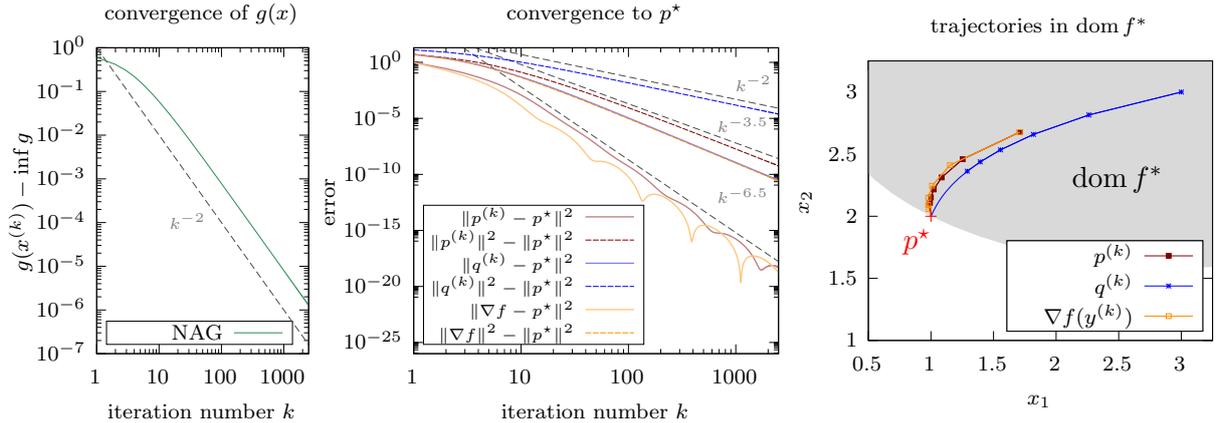

    \begin{minipage}{0.25\linewidth}
        \centering
        {\scriptsize\input{ellipsoid_g}}
    \end{minipage}
    \begin{minipage}{0.39\linewidth}
        \centering
        {\scriptsize\input{ellipsoid_dual}}
    \end{minipage}
    \begin{minipage}{0.35\linewidth}
        \centering
        {\scriptsize\input{ellipsoid_trajectory}}
    \end{minipage}
    \caption{\label{fig:ellipsoid} Convergence behavior of accelerated gradient method~\eqref{eqn:NAG-method-fixed-parameter} applied to unbounded $f$~\eqref{eqn:ellipsoid}.
    The leftmost plot shows the history of $g(x^{(k)}) - \inf_{x \in \RR^n}g(x)$, together with a referential line of $k^{-2}$.
    The central plot shows the convergence behavior of $p^{(k)}$, $q^{(k)}$, and $\nabla f(y^{(k)})$ using two measures: $\|\bullet -~p^\star\|^2$ and $\|\bullet\|^2 - \|p^\star\|^2$, together with referential lines of $k^{-2}$, $k^{-3.5}$, and $k^{-6.5}$.
    $\nabla f$ in the legend denotes $\nabla f(y^{(k)})$.
    Note that the curves of $\|q^{(k)} - p^\star\|^2$ and $\|\nabla f(y^{(k)})\|^2 - \|p^\star\|^2$ are almost overlapping.
    The rightmost plot illustrates the trajectories of $p^{(k)}$, $q^{(k)}$, and $\nabla f(y^{(k)})$; the first six points are marked to show their speeds of approach.
    }
\end{figure}

Convergence behavior of $p^{(k)}$, $q^{(k)}$, $\nabla f(y^{(k)})$, and $g(x^{(k)})$ is shown in \cref{fig:ellipsoid}.
The convergence of $g(x^{(k)})$ appears to be $\Theta(k^{-2})$.
The convergence of $p^{(k)}$, $q^{(k)}$, and $\nabla f(y^{(k)})$ appears to be polynomial, as illustrated by the referential lines.
In particular, $p^{(k)}$ and $\nabla f(y^{(k)})$ converge fastest among these three, apparently at a rate of $\|\bullet -~p^\star\|^2 = \mathcal{O}(k^{-6.5})$.

\section{Concluding Remarks}\label{sec:conclusion}
In this paper, we have analyzed the behavior of gradient descent, Nesterov's accelerated gradient method, and its continuous-time model, applied to lower-unbounded convex functions.
In all of them, the normalized negative velocity $p$ and the normalized negative displacement $q$ both converge to the minimum-norm point $p^\star$ of $\dom f^\ast$ at $\mathcal{O}(k^{-1})$, $\mathcal{O}(k^{-2})$, $\mathcal{O}(t^{-2})$, respectively.
For gradient descent and the NAG ODE, these results can be understood from a viewpoint of the mirror descent setting in the dual space; gradient descent is equivalent to mirror descent for the dual norm-minimization problem, and the NAG ODE is equivalent to the accelerated mirror descent ODE for the dual problem.
However, for Nesterov's accelerated gradient method, there seems to be no straightforward relation to existing discrete-time accelerated mirror descent methods~\cite{DO2018,DSZ2024,KBB2015,Nesterov2005}.
Finding or constructing such an accelerated mirror descent method would be an open problem.

Whereas $\|p^{(k)}\|^2 - \|p^\star\|^2 = \mathcal{O}(k^{-2})$ and $\|q^{(k)}\|^2 - \|p^\star\|^2 = \mathcal{O}(k^{-2})$ are almost tight (\cref{eg:tight}), it might be possible to improve the convergence rates of $\|p^{(k)} - p^\star\|^2$ and $\|q^{(k)} - p^\star\|^2$.
Indeed, our numerical results show that $\|\bullet -~p^\star\|^2$ can converge much faster than $\|\bullet\|^2 - \|p^\star\|^2$.
Another supporting fact on this improvement is \cite[Theorem~2.2.4]{Nesterov2018}, which states that in a different variant of NAG, if $f$ has a minimizer $x^\star$, then a convex combination $g_k$ of $\{\nabla f(y^{(j)})\}_{0 \le j \le k - 1}$ satisfies $\|g_k\|^2 = \mathcal{O}(\|x_ 0 -x^\star\|^2k^{-4})$.
However, in the unbounded setting, a minimizer of $g$ usually does not exist (see the end of \cref{sec:preliminaries}); therefore, it is not clear whether we can expect a corresponding $\mathcal{O}(k^{-4})$ convergence to hold.

The numerical results also suggest that $\nabla f(y^{(k)})$ converges to $p^\star$.
Indeed, \cite{CSY2022,SDJS2022} show that if $f$ has a minimizer, then $\min_{0 \le i \le k}\|\nabla f(y^{(i)})\|^2 = o(k^{-3})$ holds in another standard version of NAG.
Although their proofs are also based on energy functions, attempts to generalize them in the same way as \cref{dfn:Lyapunov_continuous,dfn:Lyapunov_discrete} do not succeed.
The difficulty lies in the $f(x^{(k)}) - f(x^\star)$ term in the increment of their energy functions; since $g$ usually has no minimizer, the nonnegativity of this term cannot be used in our setting.
For the same reason, the first coefficient in \eqref{eqn:NAG-method-fixed-parameter} is restricted to be $\frac{k + 2}{(k + 1)L}$, as opposed to arbitrary $\eta \le 1/L$ as in \cite{CSY2022,SDJS2022,SBC2016}.

\section*{Acknowledgments}
The author thanks Shun Sato, Takayasu Matsuo, Kansei Ushiyama, Hiroshi Hirai, and Masaru Ito for their valuable advice and discussions.
This work was supported by the European Union (ERC Starting Grant SYMOPTIC, 101040907) and by the Deutsche Forschungsgemeinschaft (DFG, German Research Foundation) -- Project 556164098.

\printbibliography

\appendix

\section{Calculations for \texorpdfstring{\cref{rem:discrete_specific_coefficients}}{Remark \ref{rem:discrete_specific_coefficients}}}\label{app:coefficients}
    In this appendix, we show detailed calculations for the expressions in \cref{rem:discrete_specific_coefficients}.
    We first state some finite sum results, which can be easily checked:
    \begin{gather*}
        \sum_{i = 1}^k\frac1i \le 1 + \log k,\qquad \sum_{i = 1}^ki(i + 1) = \frac{k(k + 1)(k + 2)}3,\\
        \sum_{i = 1}^ki^2(i + 1) = \frac{k(k + 1)(k + 2)(3k + 1)}{12},\qquad \sum_{i = 1}^ki(i + 1)^2 = \frac{k(k + 1)(k + 2)(3k + 5)}{12}.
    \end{gather*}
    Then, under $A_k \coloneqq k(k + 1)/L$, we obtain the following expressions for $P_k, Q_k$, and $C_k$:
    \begin{gather*}
        P_k = \frac{Lk(k + 1)^2(k + 2)}{(k + 1)\sum_{i = 1}^ki(i + 1)^2} = \frac{12L}{3k + 5},\qquad
        Q_k = \frac{2Lk(k + 1)}{\sum_{i = 1}^ki^2(i + 1)} = \frac{24L}{(k + 2)(3k + 1)},\\
        C_k = \frac{2L(k + 1)(k + 2)}{\sum_{i = 1}^ki(i + 1)^2} = \frac{24L}{k(3k + 5)}.
    \end{gather*}
    By using $\sqrt{i(i+ 1)} \le i + 1$, we have
    \begin{gather*}
        \begin{aligned}
            \sqrt{B_k} &= \sqrt{2L}\frac{k(k + 1)\sqrt{(k + 1)(k + 2)} + 2\sum_{i = 1}^ki\sqrt{i(i + 1)}}{\sum_{i = 1}^ki(i + 1)^2}\\
            &\le \sqrt{2L}\frac{k(k + 1)(k + 2) + 2\sum_{i = 1}^ki(i + 1)}{\sum_{i = 1}^ki(i + 1)^2} = \frac{20\sqrt{2L}}{3k + 5},
        \end{aligned}\\
        \sqrt{\tilde B_k} = 2\sqrt{2L}\frac{\sum_{i = 1}^ki\sqrt{i(i + 1)}}{\sum_{i = 1}^ki^2(i + 1)} \le 2\sqrt{2L}\frac{\sum_{i = 1}^ki(i + 1)}{\sum_{i = 1}^ki^2(i + 1)} = \frac{8\sqrt{2L}}{3k + 1}.
    \end{gather*}
    Then, $C'_k$ is bounded by
    \[
        C'_k = B_k + 2C_k \le \frac{800L}{(3k + 5)^2} + \frac{48L}{k(3k + 5)} \le \frac{800L}{3k(3k + 5)} + \frac{48L}{k(3k + 5)} = \frac{944L}{3k(3k + 5)}.
    \]
    Using $c \coloneqq \frac{\langle\nabla f(x_0) - p^\star, p^\star\rangle}{4LD_f(x_0, p^\star)}$, $\tilde C_k$ is bounded by
    \[
        \tilde C_k = \frac{Lk(k + 1)}{2\sum_{i = 1}^ki^2(i + 1)}\left(8c + 8\sum_{i = 1}^{k - 1}\frac1i\right) \le \frac{48L(c + 1 + \log k)}{(k + 2)(3k + 1)}.
    \]
    Then, $\tilde C'_k$ is bounded by
    \begin{align*}
        \tilde C'_k = \tilde B_k +2\tilde C_k
        &\le \frac{128L}{(3k + 1)^2} + \frac{96L(c + 1 + \log k)}{(k + 2)(3k + 1)}\\
        &\le \frac{96L}{(k + 2)(3k + 1)} + \frac{96L(c + 1 +\log k)}{(k + 2)(3k + 1)} = \frac{96L(c + 2 + \log k)}{(k + 2)(3k + 1)},
    \end{align*}
    where the second inequality follows from $\frac4{3k + 1} \le \frac3{k + 2}$ for $k \ge 1$.

\end{document}